\documentclass[reqno,10pt]{amsart}
\usepackage{esint}

\usepackage{amsmath}
\usepackage{amssymb}
\usepackage{amsfonts}
\usepackage{graphicx}
\usepackage[abs]{overpic}
\usepackage{caption}
\usepackage{wrapfig}
\usepackage{float}
\usepackage{graphicx}
\usepackage{mathrsfs}
\usepackage{accents}
\usepackage{amsthm}
\usepackage{hyperref}
\usepackage{mathtools}

\usepackage{tikz-cd}
\usepackage[colorinlistoftodos]{todonotes}
\usetikzlibrary{matrix,decorations.text}
\usepackage{adjustbox}

\usepackage{bbm}
\usepackage{dsfont}

\usepackage{enumitem}

\usepackage{nccmath}

\usepackage{stackengine}
\newcommand{\cdddot}{\mathrel{\Shortstack{{.} {.} {.}}}}

\def\XXint#1#2#3{{\setbox0=\hbox{$#1{#2#3}{\int}$ }
\vcenter{\hbox{$#2#3$ }}\kern-.6\wd0}}

\setlist[enumerate,1]{font=\normalfont}
\setlist[itemize,1]{font=\normalfont}

\newlist{thmlist}{enumerate}{1}
\setlist[thmlist]{label=(\roman{thmlisti}),
	ref=(\roman{thmlisti}),font=\normalfont,
	noitemsep}

\newtheorem{theorem}{Theorem}[section]
\newtheorem{lemma}[theorem]{Lemma}

\newtheorem{definition}[theorem]{Definition}

\newtheorem{rem}[theorem]{Remark}

\newcommand*{\di}{\mathop{}\!\mathrm{d}}

\newcommand{\N}{\mathbb{N}}

\newcommand{\R}{\mathbb{R}}

\DeclareMathOperator*{\argmin}{arg\,min}

\newcommand{\Id}{\mathbf{Id}}
\newcommand{\id}{\mathbf{id}}
\newcommand{\eps}{\varepsilon}

\newcommand{\defas}{\coloneqq}
\newcommand{\sym}{\mathrm{sym}}

\newcommand{\aC}{C_0}
\newcommand{\ac}{c_0}

\def\Id{\mathbf{Id}}
\def\id{\mathbf{id}}

\def\eps{\varepsilon}

\def\dist{\operatorname{dist}}

\def\XXint#1#2#3{{\setbox0=\hbox{$#1{#2#3}{\int}$}
     \vcenter{\hbox{$#2#3$}}\kern-.5\wd0}}

\DeclareMathOperator{\diver}{div}

\DeclarePairedDelimiterX\setof[1]\{\}{#1}
\DeclarePairedDelimiterX\abs[1]\lvert\rvert{#1}
\DeclarePairedDelimiterX\norm[1]\lVert\rVert{#1}
\DeclarePairedDelimiterX\sprod[2]\langle\rangle{#1, #2}

\usepackage{color}

\newcommand{\REV}{\color{black}}
\newcommand{\END}{\color{black}}
\newcommand{\REVB}{\color{black}}

\numberwithin{equation}{section}

\title[Nonlinear relations of viscous stress and strain rate in nonlinear Viscoelasticity]{Nonlinear relations of viscous stress and strain rate in nonlinear Viscoelasticity}

\author{Lennart Machill}


 \address[Lennart Machill]{
	Institute for Applied Mathematics, \\
	University of Bonn, \\
	Endenicher Allee 60, D-53115 Bonn, Germany \\
	https://orcid.org/0000-0002-6477-1467 \\
  }
  \email{lmachill@uni-bonn.de}

\begin{document}

\maketitle

\begin{abstract}
We consider a Kelvin-Voigt model for viscoelastic second-grade materials, where the elastic and the viscous stress tensor both satisfy \REV frame-indifference. \END  Using a rigidity estimate by {\sc Ciarlet and Mardare} \cite{CiarletMardare}, existence of weak solutions is shown by means of a frame-indifferent time-discretization scheme. Further, the result includes viscous stress tensors which can be calculated by nonquadratic polynomial densities. Afterwards, we investigate the long-time behavior of solutions in the case of small external loading and initial data. Our main tool is the abstract theory of metric gradient flows \cite{AGS}.
\end{abstract}

\section{Introduction}

The Kelvin-Voigt rheology is a classical concept in engineering science and continuum mechanics. 
It serves as a tool for describing the evolution of viscoelastic solids, where slow continuous deformations are observed, tending to a recoverable configuration of the material.
The key feature of  this model is the stress tensor, which is determined by summation of the elastic and the viscous stress tensor.  
\REV Relying on Hooke's law, these tensors are typically assumed to depend linearly on the displacement gradient and its time derivative, respectively. \END
However, this relation may break down if the deformation is rather large and if the undeformed and deformed configurations are significantly different.
In this article, we focus on a version of the Kelvin-Voigt model for large deformations and address this latter effect. 
More precisely, we consider the balance of momentum for second-grade
nonsimple materials without inertia, which is governed by the system of
equations 
\begin{align}\label{eq:viscoel-nonsimple}
-{\rm div}\Big( \partial_F W(x, \nabla y)  -   {\rm div} (\partial_GP(x,\nabla^2 y)) + \partial_{\dot{F}}R(x,\nabla y,\partial_t \nabla y)  \Big) =  f        \ \ \  \text{ in $ [0,T] \times    \Omega$}.
\end{align}
Here, $[0, T]$ is a process time interval for some   time horizon  $T > 0$, $\Omega \subset \R^d$ is a bounded domain representing the reference configuration, and  $y \colon [0, T] \times \Omega \to \R^d$ indicates a Lagrangian \emph{deformation}. By $f \colon [0, T] \times \Omega \to \R^d$ we indicate a volume density of \emph{external forces}  acting on $\Omega$. 
The elastic stress tensors can be calculated by the densities $W \colon \Omega \times \R^{d\times d}   \to \R \cup \setof{+\infty}$ and $P \colon\Omega \times \R^{d \times d \times d} \to [0,\infty)$, where $F$ and $G$ are the place-holders for the gradient and Hessian of $y$, respectively.   $R \colon \Omega \times \R^{d \times d} \times \R^{d \times d}   \to \R $ is the density of the viscous stress tensor, depending on both strain and strain rate with $\dot F$ being the place-holder of the strain rate.
Respecting the fundamental concept of \REV frame-indifference \END in nonlinear continuum mechanics, we are interested in densities $W$ and $R$ that can be rewritten in terms of the right Cauchy-Green tensor $C = F^TF$ and its derivative in time $\dot C = \dot F^T F + F^T \dot F$, i.e., we assume that there exist nonnegative functions $\hat W$ and $\hat R$ such that
\begin{equation*}
 W(x,F) = \hat W (x,C) \qquad \text{and} \qquad  R(x,F,\dot F) = \hat R(x,C,\dot C).
\end{equation*} 
For a thorough discussion of the latter principle in the context of elastic and viscous stresses, we refer to {\sc Antman} \cite{Antmann04Nonlinear,Antmann98Physically}. 
 $P$ corresponds to a second-order regularizing term, and a precise definition of the corresponding stress tensor is provided in \eqref{definitiondivergencesecondorder} below.
 
 \REV
We start by giving an overview of the existence theory for the underlying equations without the second gradient.
Given initial data appropriately close to a smooth equilibrium, existence of weak solutions was proven by {\textsc{Potier-Ferry}} in \cite{potier-ferry-1,potier-ferry-2}.
At present, there is no existence theory for \emph{frame-indifferent} viscous stresses at large strains since this condition leads to the loss of monotonicity in the strain rate, see e.g.\ \cite[Appendix~B]{demoulini} and \cite[Section~4(b)]{sengul2}.
Nevertheless, under different monotonicity assumptions significant progress has been made in the first quarter of this century \cite{demoulini,Lewick,Tvedt}.
Moreover, we highlight that the Poynting–Thomson model, which extends viscoelasticity beyond the Kelvin-Voigt rheology by including both Maxwell and Kelvin-Voigt elements, has been studied, proving the existence of energetic solutions \cite{Poynting}.
\END

To date, existence results of weak solutions can only be guaranteed \REV for second-grade materials, \END which is a concept going back to  {\sc Toupin} \cite{Toupin62, Toupin64}. In general, this principle indeed proved to be useful in mathematical continuum mechanics, see e.g.~\cite{BallCurrieOlver81Null, Batra76Thermodynamics, HealeyKroemer:09, MielkeRoubicek16Rateindependent, Podio02Contact}.
If $R$ is quadratic in $\dot F$, i.e., viscous stress tensor and strain rate are related \emph{linearly}, existence of weak solutions to \eqref{eq:viscoel-nonsimple} has been shown in \cite{RBMFMK, positivity24, MielkeRoubicek}, including an additional coupling with a nonlinear heat equation \REV in connection with the laws of thermodynamics. \END This approach has been extended in various directions over the last years, including models allowing for self-contact   \cite{gravina, Kroemer},  a nontrivial coupling with a diffusion equation \cite{liero} and a Cahn-Hillard equation \cite{Leonie}, inertial effects \cite{thermoelasto,Schwarzacher}, homogenization \cite{gahn}, or applications to fluid-structure interactions \cite{Schwarzacher}.
For the derivation of effective theories for (thin) materials, we refer to \cite{RBMFMK,positivity24,RBMFLM,MFMK,MFMKDimension,MFLMDimension2D1D,MFLMDimension3D1D,Oosterhout}.
While the aforementioned results are formulated using the Lagrangian approach, several recent works employ the alternative Eulerian perspective instead, see \cite{Roubicek23Eulerian,Roubicek23Eulerian2,Roubicek23Eulerian3,Roubicek23Eulerian4}. 
Eventually, a comprehensive overview of the model itself and further related results can be found in \cite{sengul}.

While previous existence theories rely on a linear relation between viscous stress and strain rate, to the best of our knowledge, existence has not been proved in the nonquadratic case, except for a result neglecting frame-indifference \cite{Bulicek}. The simplest choice for such a potential is $R(x,F,\dot F) = {\tilde{p}}^{-1}\vert \dot F^T F + F^T \dot F\vert^{\tilde{p}} $ for some $\tilde{p} \neq 2$, leading to the viscous stress tensor
\begin{align}\label{viscousstresstensor}
\partial_{\dot{F}}R(x,\nabla y,\partial_t \nabla y)  = 2 \nabla y \vert   (\partial_t \nabla y)^T \nabla y + (\nabla y)^T  \partial_t \nabla y\vert^{\tilde{p}-2} \big( (\partial_t \nabla y)^T \nabla y + (\nabla y)^T \partial_t \nabla y  \big) .
\end{align}
 In the previously mentioned results, one shows the existence of weak solutions by relying on a Minty-trick for monotone operators to handle the nonlinearity induced by $P$. However, it is not straightforward to apply this argument in the presence of an additional nonlinear viscous stress tensor as in \eqref{viscousstresstensor}.
The first main theorem of this article addresses the existence of weak solutions. By introducing a new application of the rigidity estimate by {\sc Ciarlet and Mardare} \cite{CiarletMardare}, which itself is based on the rigidity estimate by {\sc Friesecke, James and Müller} \cite{FrieseckeJamesMueller:02}, we circumvent the latter issue by making use of a gradient flow formulation of \eqref{eq:viscoel-nonsimple} in a metric space \REV \cite{MFMK,MOS}. \END

We now highlight the main strategy and challenges of this problem by treating the specific choice of $R$ in \eqref{viscousstresstensor} and setting $f = 0 $ for the sake of simplicity. Fixing a timestep $\tau>0$ such that $T\tau^{-1} \in \N$ and an initial deformation $y^0$, we inductively solve the minimization problems
\begin{align}\label{eq: disc2}
y^n \in \argmin\limits_y \frac{1}{\tilde p \tau^{\tilde p -1} } \int_\Omega \vert (\nabla y^{n-1})^T \nabla y^{n-1} - (\nabla y)^T \nabla y \vert^{\tilde p} \di x + \int_\Omega W(x,\nabla y) \di x + \int_\Omega P(x,\nabla^2 y) \di x
\end{align}
for $n = 1, \ldots, T \tau^{-1}$. 
Then, at least formally, limits of suitably defined interpolations of the time-discrete solutions converge to solutions of \eqref{eq:viscoel-nonsimple} as $\tau \to 0$. 
Putting this approach into a rigorous framework, we show convergence of time-discrete approximations to so-called \emph{curves of maximal slope}. Curves of maximal slope satisfy a certain energy-dissipation-balance which allows to relate these curves to weak solutions of the system \eqref{eq:viscoel-nonsimple}. In the limiting passage $\tau \to 0$, 
the crucial difficulty is proving a suitable bound on $y^{n} - y^{n-1}$. To overcome this issue, the rigidity estimates in \cite{CiarletMardare} yield the bound
\begin{align}\label{Ciarletintro}
\int_\Omega \vert \nabla y^{n-1} - \nabla y^n \vert^{\tilde p} \di x \leq C \int_\Omega \vert (\nabla y^{n-1})^T \nabla y^{n-1} - (\nabla y^n)^T \nabla y^n \vert^{\tilde p}  \di x,
\end{align}
where the constant $C>0$ depends on $\nabla y^{n-1}$, but can be chosen uniformly among functions in the set $K \defas \{ y \in W^{2,p}(\Omega;\R^d) : \det \nabla y > \mu , \ \Vert \nabla y \Vert_{C^\alpha(\Omega;\R^{d \times d} ) }< \mu^{-1} \}$ for $\mu>0$. 
Indeed, as the scheme provides a bound on $ \int_{\Omega}   W(x, \nabla y^n(x)) +  P(x, \nabla^2 y^n(x))   \, {\textrm d}x  $ which is uniform in $n $, suitable growth conditions on $W$ and $P$ then imply that $(y^n)_n$ lies in $K$. Notice that this argument fails in the absence of $P$.

Whereas convergence of time-discrete approximations to curves of maximal slope has already been shown \REV for quadratic dissipation potentials \END in the case of small external loading and initial data \cite{MFMK,MFLMDimension3D1D}, the corresponding result for large strains is new and established for $\tilde p \in \REVB (1,+\infty)$, \END see Theorem~\ref{maintheorem1}.
Comparing \eqref{eq: disc2} with the discretization schemes in \cite{RBMFMK,MielkeRoubicek}, a further advantage is that the functional minimized in \eqref{eq: disc2} is \REV frame-indifferent, \END i.e., this important feature of the model is already satisfied on the time-discrete level. We prefer this approximation, as otherwise rigid body motions generate dissipation, suggesting an incorrect physical choice.
 \REV This dissipation distance might also be applicable for quadratic dissipation potentials in the presence of inertial effects, by following the variational approximation scheme developed by \cite{Schwarzacher}. It would be interesting if more general polynomial viscous densities can be used in this context. However, the adaptation of the gradient flow structure to such a problem of hyperbolic type seems to be challenging.  \END
% Following this strategy, we prove an existence result for small strains in the case $\tilde p \in (2,+\infty)$, see Theorem~\ref{maintheorem1}(ii). 
%If one tries to recover the previous result also for large deformations, one needs to prove a suitable representation of the local slope in order to show lower semicontinuity of the local slope, a quantity that arises in the setting of metric gradient flows. This relies on specific convexity properties, and if $\tilde p>2$, the remainder of the Taylor expansion of the elastic energy cannot be controlled with the choice of our metric, see Lemma~\ref{lem:localconv} and Lemma~\ref{lem:stronguppergradient3d}.

So far, the long-time behavior of solutions for large strains has been investigated in one dimension \cite{AndrewsBall,Ballsengul}, and, to the best of our knowledge, results in higher dimensions are currently unavailable.
The second main theorem of this article addresses the long-time behavior of weak solutions in the case of \REV \emph{small} \END  external loading and initial data.
 We show that for $\tilde p \in (1,2)$ and $\tilde p = 2$ the elastic energy decays polynomially and exponentially, respectively, while
for $\tilde p > 2$ equilibrium is reached in finite time. Here, we benefit from the formulation as a metric gradient flow, inspired by the abstract theory in \cite{AGS}.
For small strains, one can show $\lambda$-convexity of the elastic energy with $\lambda>0$ along convex combinations of deformations, while the metric satisfies a generalized convexity property, see Lemma~\ref{lem:localconv}(iii).
Then, the resulting representation of the local slope allows for a control on the time-derivative of the elastic energy with the elastic energy itself. 
Heuristically, if both $\nabla y$ and $\partial_t \nabla y$ in \eqref{eq:viscoel-nonsimple} depend linearly on \emph{constant} tensors $\partial_FW$ and $\partial_{\dot F} R$, respectively, it is well known that solutions decay exponentially, i.e., our result is consistent with the linearized case.

The plan of the paper is as follows. In Section~\ref{section:2}, we introduce the model in more detail and state our main results.
Section~\ref{sec: auxi-proofs} recalls the relevant definitions about curves of maximal slope and presents an abstract theorem concerning the convergence of time-discrete solutions to curves of maximal slope.   
This abstract theory is applied in Section~\ref{sec:3dproperties}, where we prove the existence of weak solutions to \eqref{eq:viscoel-nonsimple}.
Finally, the long-time behavior is addressed in Section~\ref{sec:longtime}.

\subsection*{Notation} 

In what follows, we use standard notation for Lebesgue spaces, $L^p(\Omega)$, which are measurable maps on $\Omega\subset\R^d$, $d \in \N$, integrable with the $p$-th power (if $1\le p<+\infty$) or essentially bounded (if $p=+\infty$). Sobolev spaces,  written $W^{k,p}(\Omega)$, denote the linear spaces of  maps  which, together with their weak derivatives up to the order $k\in\N$, belong to $L^p(\Omega)$. 
    Moreover, for a function $\hat v \in W^{k,p}(\Omega)$ the set $W^{k,p}_{\hat v}(\Omega)$ contains maps from $W^{k,p}(\Omega)$ having boundary conditions (in the sense of traces) up to the   $(k-1)$-th  order with respect to $\hat v$. 
If the target space is a Banach space $E \neq \R$, we use the usual notion of Bochner-Sobolev spaces, written $W^{k,p}(\Omega;E)$.  For more details on Sobolev spaces and their   duals,   we refer to \cite{AdamsFournier:05}. $\nabla$ and $\nabla^2$ denote the spatial gradient and Hessian, respectively, and $\partial_t$ indicates a time derivative. Further, $\Id \in \R^{d \times d}$ is the identity matrix. Finally, $\vert A \vert$   stands for   the Frobenius norm  of a matrix $A \in \R^{d\times d}$,   and ${\rm sym}(A) = \frac{1}{2}(A^T + A)$ and ${\rm skew}(A) = \frac{1}{2}(A - A^T)$ indicate the symmetric and skew-symmetric part, respectively.  
\REV Finally, $c_0>0$ and $C_0 >0$ are positive constants, while $C>0$ is a generic constant which may vary from line to line. \END

 \section{The model  and main results}\label{section:2}
 \subsection{The variational setting}\label{section:variational}

 In this section, we describe the model and discuss the variational setting.  Following the discussion in \cite[Section 2.2]{MOS} and \cite[Section 2]{MFMK}, we model \eqref{eq:viscoel-nonsimple} as a metric gradient flow. For this purpose, we need to specify three main   ingredients:  the state space that contains admissible deformations of the material, the    elastic energy that drives the evolution, and  the dissipation mechanism represented by a distance.
 
 \subsection*{Elastic energy}

Let $\Omega\subset \R^d$ be an open, bounded set with Lipschitz boundary, denoting the reference configuration of the material. 
%Let $\Gamma_D\subset \Gamma$ be an open set, and $\Gamma_N\defas \Gamma \setminus \Gamma_D$,   representing \emph{Dirichlet and Neumann parts} of the boundary, respectively.
We define the \emph{elastic energy} associated with a deformation $y\colon \Omega \to \R^d$ by
 \begin{align}\label{energyphieps}
 	\phi(y) = \int_{\Omega} W(x,\nabla y(x)) \,  {\rm d}x \ +  \int_{\Omega} P(x, \nabla^2 y(x)) \,   {\rm d}x  -  \int_{\Omega} f (x)\cdot \, y(x) \,  {\rm d}x .
	% -  \int_{\Gamma_N} g_\eps(x)\cdot \, y(x) \,  {\rm d}\mathcal{H}^{d-1}.
 \end{align}
  Here, $W\colon \R^d \times \R^{d \times d} \to [0,\infty]$ denotes  a   frame-indifferent stored energy density with the usual assumptions in nonlinear elasticity. More precisely,  we  suppose that 
\begin{enumerate}[label=(W.\arabic*)]
  \item \label{W_regularity}  $W$  is  $C^3$ on $\R^d \times \REV GL^+(d) \END$,
  \item \label{W_frame_invariace} $W$ is \REV frame-indifferent, \END i.e., $W(x,QF) = W(x,F)$ for all $F \in GL^+(d)$ and $Q \in SO(d)$,
  \item \label{W_lower_bound}  $W(x,F) \ge \ac \big(|F|^2 + \det(F)^{-q}\big) - \aC$ for all $F \in GL^+(d)$, where $q \ge \frac{pd}{p-d}$,
\end{enumerate}
where \REV $\REV GL^+(d) \END := \lbrace F \in \R^{d \times d}: \det F>0 \rbrace$ and $SO(d) = \lbrace Q\in \R^{d \times d}: Q^T Q = \Id, \, \det Q=1 \rbrace$. \END
%A standard example for the energy density $W$ would $W(x, F) \defas g(x)\dist^2(F,SO(3)) + \det(F)^{-q } $, where $g$ denotes an $1$-periodic function on $\R^d$. 
Besides the elastic energy density $W$ depending on the deformation gradient, we also consider a \emph{strain-gradient energy term} $P$ depending on the Hessian $\nabla^2 y$, adopting the concept of second-grade nonsimple materials, see \cite{Toupin62}. More specifically, for $p>d$, let $P\colon \R^d \times \R^{d\times d \times d} \to [0, +\infty) $ satisfy the following conditions:
 \begin{enumerate}[label=(H.\arabic*)]
  \item \label{H_regularity} $P$ is convex and $C^1$ on $\R^d \times \R^{d \times d\times d}$,
  \item \label{H_frame_indifference} Frame indifference: $P(x,QG) = P(x,G)$ for all $G \in \R^{d \times d \times d}$ and $Q \in SO(d)$,
  \item \label{H_bounds} $\ac \abs{G}^p \leq P(x,G) \leq \aC  \abs{G}^p $ and  $ \vert \partial_G P(x,G) \vert \leq C_0 \vert G \vert^{p-1} $  for all $G \in \R^{d \times d \times d}$.
\end{enumerate}
 The remaining term $f \in L^\infty(\Omega)$ denotes a dead load. 

 \subsection*{State space} 
Given $M>0$ and the Dirichlet boundary conditions $\hat y \in W^{2,p}(\Omega;\R^d)$, we introduce the set of admissible configurations by
 \begin{align}\label{assumption:clampedboundary}
 	\mathscr{S}_{M} := \Bigg\{ y \in W^{2,p}(\Omega;\R^d) : y = \hat y
  \quad {\rm on } \quad \partial \Omega , \quad \phi (y) \leq M \Bigg\}.
 \end{align} 
  In the next subsection, we will assign a suitable metric to $\mathscr{S}_{M}$ which simultaneously is related to the viscous stress.
 
 \subsection*{Dissipation mechanism:} 
% Consider now  time-dependent deformations $w\colon[0,T]\times \Omega \to \R^3$. In contrast to elasticity, viscosity is not only related to the strain $\nabla w$ but also to the strain rate $\partial_t \nabla  w$.
%  It can be expressed in terms of $R(x,\nabla w, \partial_t \nabla w)$ for the \emph{dissipation potential} $R\colon \R^d \times \R^{3 \times 3} \times \R^{d \times d} \to [0,\infty)$ given in \eqref{eq:viscoel-nonsimple}. 
Inspired by \cite[Example~2.4]{MOS}, we consider the function
\begin{align}\label{explicitchoiceofD}
D(x,F_1,F_2) := \vert A(x)F_1^TF_1 - A(x)  F_2^T F_2 \vert,
\end{align}
where $A(x)\in \REV GL^+(d) \END $ is an invertible matrix,
and  the following density of the viscous stress tensor
\begin{align}\label{defRx}
R(x, F,\dot F) := \frac{1}{{\tilde p}} \vert A(x) \dot C \vert^{\tilde p},
\end{align}
where $\dot C = \dot F^T F + F^T \dot F \in \R^{d \times d}_{\rm sym}$ \REV and  $\tilde p \in (1,+\infty)$. \END
A computation of the derivative with respect to $\dot F$ yields
\begin{align}\label{explicitex}
 \partial_{\dot{F}}R(x,F, \dot F) = \vert A(x)\dot C \vert^{\tilde p - 2}   \big( F\dot C A(x)^T A(x)  +   F A(x)^T A(x)\dot C \big).
\end{align}
Indeed, using Einstein's summation convention, the chain rule, and the symmetry of $\dot C$, we get
\begin{align*}
 \partial_{\dot{F}_{ij}}R(x,F, \dot F) &=  \vert A(x)\dot C \vert^{\tilde p - 2} \big( A(x)\dot C  \big)_{sl} \partial_{\dot{F}_{ij}} \big( A(x)\dot C \big)_{sl} \\
  &= \vert A(x)\dot C \vert^{\tilde p - 2} \big( A(x)\dot C  \big)_{sl} \partial_{\dot{F}_{ij}}   A(x)_{sk} \big( \dot F_{mk} F_{ml} +  F_{mk} \dot F_{ml}  \big) \\ 
    &= \vert A(x)\dot C \vert^{\tilde p - 2} \Big( \big( A(x)\dot C  \big)_{sl} A(x)_{sj} F_{il} + \big( A(x)\dot C  \big)_{sj} A(x)_{sk} F_{ik} \Big)  \\
    &= \vert A(x)\dot C \vert^{\tilde p - 2}   \big( F\dot C A(x)^T A(x)  +   F A(x)^T A(x)\dot C \big)_{ij} 
\end{align*}
for $i, j \in \{1, \hdots, d \}$.
The coefficients of $A(x)$ then define the materials viscosity and particularly include \emph{anisotropic} viscous stress tensors.
We assume that $D$ satisfies for all $F_1,F_2\in \REV GL^+(d) \END $ 
 \begin{enumerate}[label=(D.\arabic*)]
  \item \label{D_bound} $ C_0 \, \vert F_1^T F_1 - F_2^T F_2\vert \geq D(x, F_1, F_2) \geq c_0 \, \vert F_1^T F_1 - F_2^T F_2\vert$.
\end{enumerate}
 \ref{D_bound} ensures that the Frobenius norm of $A(\cdot)$ and its inverse are uniformly bounded.
Eventually, we associate a metric $\mathcal{D}$ to the set of admissible deformations $\mathscr{S}_{M}$, defined by 
\begin{align}\label{dissipationdistance}
\mathcal{D}(y_1,y_2) = \Big(\int_\Omega  D(x ,  \nabla y_1(x),  \nabla y_2(x) )^{\tilde{p}} {\rm d}x \Big)^{1/{\tilde{p}}}.
\end{align}
for $y_1,y_2 \in \mathscr{S}_{M}$. 
     In our setting as metric gradient flows, \ref{D_bound} is particularly needed for proving
 that $(\mathscr{S}_{M},\mathcal{D})$ is a metric space, see Lemma~\ref{th: metric space}(i) below.
 %  Here, we stress once again that $D$ satisfies time-discrete frame-indifference, we refer to \cite[Lemma~2.1]{MOS} for the relation
%   Notice that $D$ and $R$ are related via
%\begin{align}
%\lim\limits_{\delta \to 0} \frac{1}{{\tilde p} \delta^{\tilde p}} D^{\tilde p}(F+\delta \dot{F}, F) 
%=  \lim\limits_{\delta \to 0} \frac{1}{{\tilde p} \delta^{\tilde p}}\left|\delta A(x) \left(\dot{F}^{\top} F+F^{\top} \dot{F}\right)+\delta^2 A(x) \dot{F}^{\top} \dot{F}\right|^{\tilde p} = \frac{1}{{\tilde p}}\left|A(x) \big(\dot{F}^{\top} F+F^{\top} \dot{F} \big) \right|^{\tilde p}.
%\end{align}
%%  \begin{align}\label{intro:R}
%% 	R(x,F,\dot{F}) := \lim_{\delta \to 0} \frac{1}{2\delta^2} D^2(x,F+\delta\dot{F},F) = \frac{1}{4} \partial^2_{F_1^2} D^2(x,F,F) [\dot{F},\dot{F}],
%% \end{align}
%where the second equality follows by a Taylor expansion. 

\subsection{Equations of nonlinear viscoelasticity}

Our first goal is to prove the existence of weak solutions to \eqref{eq:viscoel-nonsimple}. 
 The weak formulation  reads as follows.
 \begin{definition}[Weak solutions of nonlinear viscoelasticity]\label{def:weakformulation}
 We say that $y \in L^\infty([0,\infty); W^{2,p}(\Omega;\R^d) \cap W^{1,\tilde p}([0,+\infty); W^{1,\tilde p}(\Omega;\R^d) )$ is a  weak solution to \eqref{eq:viscoel-nonsimple} if $y(0) = y_0$, $y = \hat y$ on $\partial \Omega$, and if for a.e.~$t \geq 0$ we have
 \begin{align}
&\int_\Omega \big( \partial_F W(x,\nabla y(t,x)) +  \partial_{\dot{F}}R(x,\nabla y(t,x), \partial_t  \nabla y(t,x))  \big) : \nabla \varphi(x) + \partial_G P(x, \nabla^2 y(t,x)) \cdddot \nabla^2 \varphi(x) \di x \notag \\& \qquad = \int_\Omega f(x) \cdot \varphi (x) \di x \label{weaksolutiondef}
%+  \int_{\Gamma_N} g_\eps(x) \cdot \varphi (x) \di \mathcal{H}^{d-1}.
\end{align}
for all $\varphi \in W^{2,p}_{0} (\Omega; \R^d)$.
 \end{definition}
\REV Here, $\cdot$, $:$, and $\cdddot$ denote the scalar product between first, second, and third order tensors, respectively. \END
The divergence of the derivative of $P$ is defined as
\begin{align}\label{definitiondivergencesecondorder}
(\diver ( \partial_G P(x,\nabla^2 y) ) )_{ij} \defas  \sum\limits_{k=1}^d \partial_k (\partial_G P(x,\nabla^2 y) )_{ijk}
\end{align}
 for $i,j \in \{1,...,d\}$, and, by using integration by parts, one can check that \eqref{eq:viscoel-nonsimple} is satisfied, provided that the weak solution is sufficiently smooth.
As in \cite{MielkeRoubicek}, the system is implicitly  complemented with the boundary conditions
\begin{align*}
   \partial_G P(\nabla^2 y) \REV : \END (\nu \otimes \nu)= 0   &\qquad \text{on } \partial \Omega,
\end{align*}
where $\nu$ denotes the outward pointing unit normal on $\partial \Omega$.
We refer to \cite{Kroemer, MielkeRoubicek} for an explanation and the derivation of this condition.
 
 \subsection{Main results}\label{sec:mainconvergenceresult}

Our first  result addresses the existence of solutions to the problem. We employ an abstract convergence result concerning metric gradient flows, more precisely for curves of maximal slope and their approximation via the minimizing movement scheme. The relevant notions about curves of maximal slope  are recalled in Section \ref{sec: auxi-proofs}. In particular, the local slope of $\phi$ with respect to $\mathcal{D}$ is denoted by    $|\partial \phi|_{\mathcal{D}}$, see Definition~\ref{main def2}.

\begin{theorem}[Existence of solutions]\label{maintheorem1}
Let $M>0$, \REVB $\tilde p \in  (1,+\infty)$, \END and $y_0 \in \mathscr{S}_{M}$. 
\begin{itemize}
\item[(i)] Then, there exists a function $y  \in L^\infty([0,+\infty);    \mathscr{S}_{M}   ) \cap W^{1,{\tilde p}}([0,+\infty);W^{1,{\tilde p}}(\Omega;\R^d))$    satisfying $y(0) = y_0$, and $y$  is a $\tilde p$-curve of maximal slope for $\phi$  with respect to  $|\partial \phi|_{\mathcal{D}}$.
%\item[(ii)] Let $\tilde p \in (2,+\infty)$ \REV and \REV assume that \ref{Sadditional} holds. \END
%Assume that \ref{W_regularity}--\ref{W_lower_bound} are complemented with the conditions $W(x,\Id) = 0$ and $W(x,F) \ge \ac \dist^2(F,SO(d))$ for all $F \in GL^+(d)$. Moreover, suppose that $\hat y = \id + \delta \hat u$ for $\hat u \in W^{2,p}(\Omega;\R^d)$ and $\delta>0$ in \eqref{assumption:clampedboundary},  and that the force satisfies $f = \delta \tilde f$ for some $\tilde f \in L^\infty(\Omega)$.
%Consider a sequence of initial data $(y^\delta_0)_\delta$,  with $y^\delta_0 \in \mathscr{S}_{M}$ and $\phi(y^\delta_0) \leq M \delta^2$ for some $M>0$.
%Then, there exists $\delta'>0$ such that for all $\delta< \delta'$ there exists a function $y^\delta  \in L^\infty([0,+\infty);    \mathscr{S}_{M}   ) \cap W^{1,{\tilde p}}([0,+\infty);W^{1,{\tilde p}}(\Omega;\R^d))$    satisfying $y^\delta(0) = y_0^\delta$, and $y^\delta$  is a $\tilde p$-curve of maximal slope for $\phi$  with respect to  $|\partial \phi|_{\mathcal{D}}$.
\item[(ii)] Then, a $\tilde p$-curve of maximal slope for $\phi$  with respect to  $|\partial \phi|_{\mathcal{D}}$ is a weak solution to \eqref{eq:viscoel-nonsimple} in the sense of Definition~\ref{def:weakformulation}.
\end{itemize} 
\end{theorem}

The following theorem addresses the long-time behavior of solutions \REV under the assumption that the initial data and external forces are small.
In this case, we suppose that 
\begin{enumerate}[label=(S.\arabic*)]
	\item \label{Sadditional}  
	the initial condition in \eqref{assumption:clampedboundary} can be rewritten as $\hat y = \id + \delta \hat u$ for $\hat u \in W^{2,p}(\Omega;\R^d)$ and $\delta>0$. Similarly, we suppose that the body force is given by $f = \delta \tilde f$ for some $\tilde f \in L^\infty(\Omega)$.
	Further, we complement the conditions \ref{W_regularity}--\ref{W_lower_bound} by assuming that  $W(x,\Id) = 0$ and $W(x,F) \ge \ac \dist^2(F,SO(d))$ for all $F \in GL^+(d)$. 
  \end{enumerate}
\END
 
\begin{theorem}[Long-time behavior for small strains]\label{maintheorem3} 
%Assume that \ref{W_regularity}--\ref{W_lower_bound} are complemented with the conditions $W(x,\Id) = 0$ and $W(x,F) \ge \ac \dist^2(F,SO(d))$ for all $F \in GL^+(d)$. Moreover, suppose that $\hat y = \id + \delta \hat u$ for $\hat u \in W^{2,p}(\Omega;\R^d)$ and $\delta>0$ in \eqref{assumption:clampedboundary},  and that the force satisfies $f = \delta \tilde f$ for some $\tilde f \in L^\infty(\Omega)$.
\REV Assume that \ref{Sadditional} holds. \END
Consider a sequence of initial data $(y^\delta_0)_\delta$,  with $y^\delta_0 \in \mathscr{S}_{M}, M = M' \delta^2$ for some $M'>0$.
Let $(y^\delta)_\delta$ be a sequence of $\tilde p$-curves of maximal slopes for $\phi$  with respect to  $|\partial \phi|_{\mathcal{D}}$. 
\begin{enumerate}
\item[(i)] Then, there exists $\delta'>0$ such that for all $\delta< \delta'$ there exists a unique minimizer $y_\infty^\delta \in \mathscr{S}_{M}$ solving
\begin{align*}
\REV y_\infty^\delta \END = \argmin  \{ \phi(y) : y \in \mathscr{S}_{M}  \}.
\end{align*}
\item[(ii)] Let $\tilde p = 2$. Then, there exists $\delta'>0$ such that for all $\delta< \delta'$ we have $y^\delta(t) \REV \rightharpoonup \END y_\infty^\delta$ \REV weakly \END in $W^{2,p}(\Omega)$ as $ t \to + \infty$, and 
\begin{align*}
0 \leq \phi(y^\delta(t)) - \phi(y^\delta_\infty) & \leq \exp ({-C t}) (\phi(y^\delta_0) - \phi(y^\delta_\infty) ) 
\end{align*}
for a constant $C>0$ independent of $t$. The dissipation can be bounded from below by
\begin{align*}
2 \int_0^t \int_\Omega   R(x,\nabla y(s), \nabla \partial_t y (s)) \di x \di s \geq (1-\exp(-Ct)) \big( \phi( y^\delta_0) - \phi(y^\delta_\infty)\big) 
\end{align*}
for $t \in [0,+\infty)$.

\item[(iii)] Let $\tilde p \in (1,2)$.
Then, there exists $\delta'>0$ such that for all $\delta< \delta'$ we have $y^\delta(t) \REV \rightharpoonup \END y_\infty^\delta$ \REV weakly \END in $W^{2,p}(\Omega)$ as $ t \to + \infty$, and 
\begin{align*}
0 \leq \phi(y^\delta(t)) - \phi(y^\delta_\infty) & \leq   \left(  \big( \phi( y^\delta_0) - \phi(y^\delta_\infty)\big)^{1-s} - C (1-s) t \right)^{\frac{1}{1-s}}
\end{align*}
for $s = \frac{\tilde p}{2 \tilde p -2} >1$ and a constant $C>0$ independent of $t$.

\item[(iv)] Let $\tilde p \in (2,+\infty)$.  Then, there exist $T_{\rm ext}>0$  and  $\delta'>0$ such that for all $\delta< \delta'$ we have
 \begin{align*}
	  y^\delta(t) = y^\delta_\infty  \text{ for } t \geq T_{\rm ext}.
\end{align*}
\end{enumerate} 
\end{theorem}

Notice that we do not quantify the rate of convergence in (iv). Its proof leads to a rate which does not match the heuristic scaling. This is discussed in Remark~\ref{remarknotsharp}.

\section{Metric gradient flows}\label{sec: auxi-proofs}
%It is often convenient to consider a weaker Hausdorff topology $\sigma$ on $\mathscr{S}$ to have more flexibility in the derivation of compactness properties (see \cite[Remark 2.0.5]{AGS}). 

\subsection{Definitions}

In this section, we  recall the relevant definitions about curves of maximal slope and present an abstract theorem concerning the convergence of time-discrete solutions to curves of maximal slope.   
We consider a   complete metric space $(\mathscr{S},\mathcal{D})$. We say a curve $y\colon (a,b) \to \mathscr{S}$ is \emph{absolutely continuous} with respect to $\mathcal{D}$ if there exists $m \in L^1(a,b)$ such that
\begin{align}\label{def:absolutecontinuity}
\mathcal{D}(y(s),y(t)) \le \int_s^t m(r) \,{\rm d}r    \ \ \   \text{for all} \ a \le s \le t \le b.
\end{align}
The smallest function $m$ with this property, denoted by $|y'|_{\mathcal{D}}$, is called \emph{metric derivative} of  $y$  and satisfies  for a.e.\ $t \in (a,b)$   (see \cite[Theorem 1.1.2]{AGS} for the existence proof)
\begin{align}\label{def:metricderivative}
|y'|_{\mathcal{D}}(t) := \lim_{s \to t} \frac{\mathcal{D}(y(s),y(t))}{|s-t|}.
\end{align}
Next, we  define the notion of a \emph{$\tilde p$-curve of maximal slope}. We only give the basic definition here and refer to \cite[Section 1.2, 1.3]{AGS} for motivations and more details.  By  $h^+:=\max(h,0)$ we denote the positive part of a function  $h$.

\begin{definition}[Upper gradients, slopes, curves of maximal slope]\label{main def2} 
	We consider a   complete metric space $(\mathscr{S},\mathcal{D})$ with a functional $\phi\colon \mathscr{S} \to (-\infty,+\infty]$.

	{\rm(i)} A function $g\colon \mathscr{S} \to [0,\infty]$ is called a strong upper gradient for $\phi$ if for every absolutely continuous curve $ y\colon  (a,b) \to \mathscr{S}$ the function $g \circ y$ is Borel and 
	$$|\phi(y(t)) - \phi(y(s))| \le \int_s^t g( y(r)) |y'|_{\mathcal{D}}(r)\,   {\rm d}r    \  \ \  \text{for all} \ a< s \le t < b.$$
	
	{\rm(ii)} For each $y \in \mathscr{S}$ the local slope of $\phi$ at $y$ is defined by 
	$$|\partial \phi|_{\mathcal{D}}(y): = \limsup_{z \to y} \frac{(\phi(y) - \phi(z))^+}{\mathcal{D}(y,z)}.$$

	{\rm(iii)} Let $\tilde p \in (1,+\infty)$. An absolutely continuous curve $y\colon (a,b) \to \mathscr{S}$ is called a ${\tilde p}$-curve of maximal slope for $\phi$ with respect to the strong upper gradient $g$ if for a.e.\ $t \in (a,b)$
	$$\frac{\rm d}{ {\rm d} t} \phi(y(t)) \le - \frac{1}{{\tilde p}}|y'|^{\tilde p}_{\mathcal{D}}(t) - \frac{1}{p'}g^{p'}(y(t))$$ for $p' = {\tilde p}/({\tilde p}-1)$.
\end{definition}

\subsection{Curves of maximal slope as limits of time-discrete solutions}
 
Consider a functional $\phi \colon \mathscr{S} \to   (-\infty,\infty]   $.
   We now    describe    the construction of time-discrete solutions for the energy $\phi$ and the metric $\mathcal{D}$.
Consider a fixed time step $\tau >0$ and suppose that an initial datum $Y^0_{\tau}$ is given. Whenever $Y_{\tau}^0, \ldots, Y^{n-1}_{\tau}$ are known, $Y^n_{\tau}$ is defined as (if existent)
\begin{align}\label{time-discretescheme}
Y_{\tau}^n = {\rm argmin}_{v \in \mathscr{S}} \mathbf{\Phi}(\tau,Y^{n-1}_{\tau}; v), \ \ \ \mathbf{\Phi}(\tau,u; v):=  \frac{1}{\tilde p \tau^{\tilde p- 1}} \mathcal{D}(v,u)^{\tilde p} + \phi(v). 
\end{align}
 Then,    we define the  piecewise constant interpolation by
\begin{align*}
\tilde{Y}_{\tau}(0) = Y^0_{\tau}, \ \ \ \tilde{Y}_{\tau}(t) = Y^n_{\tau}  \ \text{for} \ t \in ( (n-1)\tau,n\tau], \ n\ge 1.  
\end{align*}
We call  $\tilde{Y}_{\tau}$  a \emph{time-discrete solution}. 
   Our    goal is to study the limit of time-discrete solutions as $\tau \to 0$. Let  $\sigma$ denote the topology on $\mathscr{S}$  for the convergence satisfying the following conditions:   
we suppose     that      
\begin{align}\label{compatibility}
\begin{split}
z_k \stackrel{\sigma}{\to} z, &\ \  \bar{z}_k \stackrel{\sigma}{\to} \bar{z}  \ \ \  \Rightarrow \ \ \ \liminf_{k \to \infty} \mathcal{D}(z_k,\bar{z}_k) \ge  \mathcal{D}(z,\bar{z})
\end{split}
\end{align}
for all $z, \bar z \in \mathscr{S}$.
Moreover, we assume that for every sequence $(z_k)_k$, $z_k \in \mathscr{S}$, and $N \in \N$ we have
\begin{align}\label{basic assumptions2}
\begin{split}
\phi(z_k) \leq N \quad \Rightarrow \ \ \ z_k \stackrel{\sigma}{\to} z \in \mathscr{S} \quad \text{(up to a subsequence)}.
\end{split}
\end{align}
   Further,    we suppose lower semicontinuity of the energies and the slopes in the following sense: for all $z \in \mathscr{S}$ and $(z_k)_k$, $z_k \in \mathscr{S}$, we have
\begin{align}\label{eq: implication}
\begin{split}
z_k \stackrel{\sigma}{\to}  z \ \ \ \ \  \Rightarrow \ \ \ \ \  \liminf_{k \to \infty} |\partial \phi|_{\mathcal{D}} (z_{k}) \ge |\partial \phi|_{\mathcal{D}} (z), \ \ \ \ \  \liminf_{k \to \infty} \phi(z_{k}) \ge \phi(z).
\end{split}
\end{align}
%We assume that there exist $\tau^*>0$, $z^* \in \mathscr{S}_0$ and $z^*_k$ satisfying $\pi_k(z^*_k) = z^*$ and 
%\begin{align}
%\inf\limits_{k\in \N}	\inf\limits_{v \in \mathscr{S}_k}\mathbf{\Phi_k}(\tau^*,z^*_k; v) \geq c_0> -\infty
%\end{align}
%for some $c_0 \in \R$.
   For the relation  of time-discrete solutions  and    curves of maximal slope we    will    use the following result.  
\begin{theorem}\label{th:abstract convergence 2}
	Suppose that   \eqref{compatibility}--\eqref{eq: implication} hold. Moreover, assume that    $|\partial \phi|_{\mathcal{D}}$ is a  strong upper gradient for $ \phi $.  Consider a  null sequence $(\tau_k)_k$. Let   $(Y^0_{\tau_k})_k$ with $Y^0_{\tau_k} \in \mathscr{S}$  and $z_0 \in \mathscr{S}$ be initial data satisfying 
	
	\begin{align*}
	& \ \  Y^0_{\tau_k} \stackrel{\sigma}{\to} z_0 , \ \ \ \ \  \phi(Y^0_{\tau_k}) \to \phi(z_0).
	\end{align*}
	Then, for each sequence of discrete solutions $(\tilde{Y}_{\tau_k})_k$  starting from $(Y^0_{\tau_k})_k$ there exists a limiting function $z\colon [0,+\infty) \to \mathscr{S}$ such that up to a  subsequence    (not relabeled)   
	$$\tilde{Y}_{\tau_k}(t) \stackrel{\sigma}{\to} z(t), \ \ \ \ \ \phi(\tilde{Y}_{ \tau_k}(t)) \to \phi(z(t)) \ \ \  \ \ \ \ \ \forall t \ge 0$$
	as $k \to \infty$, and $z$ is a ${\tilde p}$-curve of maximal slope for $\phi$ with respect to $|\partial \phi|_{\mathcal{D}}$. In particular,  $z$ satisfies the    energy-dissipation-balance      
	\begin{align}\label{maximalslope}
	\frac{1}{{\tilde p}} \int_0^T |z'|_{\mathcal{D}}^{\tilde p}(t) \, {\rm d} t + \frac{1}{p'} \int_0^T |\partial \phi|_{\mathcal{D}}^{p'}(z(t)) \, {\rm d} t + \phi(z(T)) = \phi(z_0) \ \  \ \ \ \forall \,  T>0,
	\end{align} 
	where $p' = {\tilde p}/({\tilde p}-1)$.
	
\end{theorem} 

\begin{proof}
The statement is proved in  \cite[Theorem~2.3.3]{AGS}.
\end{proof}
 
%The statement is a combination of convergence results for curves of maximal slope, see e.g.\ \cite{S2}, with their approximation by time-discrete solutions via the minimizing movement     scheme.     

\section{Properties of the model} \label{sec:3dproperties} 
First, we state properties of functions lying in the space $\mathscr{S}_{M}$.
\begin{lemma}[Bounds in $\mathscr{S}_{M}$]
Let $M>0$. There exists a constant $C= C(M, \Omega)>0$ such that for all $y \in \mathscr{S}_{M}$ it holds that
\begin{align}\label{aprioribounds}
\Vert y \Vert_{W^{2,p}(\Omega) } + \Vert y \Vert_{C^{1,1-d/p}(\Omega)} + \Vert (\nabla y)^{-1} \Vert_{C^{1-d/p}(\Omega) }\leq \REV C , \END \qquad \inf_{x \in \bar \Omega} \det ( \nabla y(x) ) \geq \REV C^{-1}. \END
\end{align}
\end{lemma}
\begin{proof}
The proof can be found in \cite[Theorem 3.1]{MielkeRoubicek} and relies on \cite{HealeyKroemer:09}. 
\end{proof}

We now provide the rigidity estimate that gives a lower bound on the metric in Lemma~\ref{th: metric space}. This inequality is the crucial ingredient for the existence of a curve of maximal slope for large strains and has not been applied yet in the context of viscoelasticity.
\begin{lemma}[Rigidity estimate]\label{nonlinearKornappl}
 For any $y_0$, $y_1 \in \mathscr{S}_{M}$ there exists a constant $C>0$, only depending on $M>0$ and $\Omega$ such that 
\begin{align}\label{Kornineq}
\Vert \nabla y_1 - \nabla y_0 \Vert_{L^{\tilde p} (\Omega) } \leq C  \Vert (\nabla y_1)^T \nabla y_1 - (\nabla y_0)^T \nabla y_0 \Vert_{L^{\tilde p} (\Omega) }.
\end{align}
\end{lemma}
Lemma~\ref{nonlinearKornappl} is essentially addressed in \cite[Theorem~3]{CiarletMardare}, and a generalization of the rigidity estimate by {\sc Friesecke, James and Müller} \cite{FrieseckeJamesMueller:02}.
The crucial difference to \cite{CiarletMardare} is that the constant $C>0$ depends on $y_0$, but can be chosen uniformly among functions in the space $ \mathscr{S}_{M}$. This has already been exploited in the case of the generalized version of Korn's inequality in \cite[Theorem~3.3]{MielkeRoubicek}.

\begin{proof} 
	Due to \eqref{aprioribounds}, we can use \cite[Theorem~3]{CiarletMardare}, yielding
	\begin{align}\label{Kornineqwithbadconstant}
	\Vert \nabla y_1 - \nabla y_0 \Vert_{L^{\tilde p} (\Omega) } \leq C  \Vert (\nabla y_1)^T \nabla y_1 - (\nabla y_0)^T \nabla y_0 \Vert_{L^{\tilde p} (\Omega) },
	\end{align}
	where $C= C(y_0)>0$ depends on $y_0$.
	The only thing left to show is that the constant can be chosen uniformly among functions in the set
	\begin{align*}
	W_{{ \REV \tilde C \END }} \defas \{ y_0 \in C^{1,1-d/p}(\Omega; \R^d) : \min\limits_{x \in \Omega} \det (\nabla y_0) \geq { \REV \tilde C \END }^{-1}, \ \Vert y_0 \Vert_{C^{1,1- d/p}(\Omega;\R^d)} \leq { \REV \tilde C \END } \}.
	\end{align*}
	Due to Arzelà-Ascoli, $W_{{ \REV \tilde C \END }}$ is a compact subset  of the space $\big( C^1(\overline{\Omega};\R^d), \Vert \cdot \Vert_{L^\infty(\Omega)} + \Vert\nabla (\cdot) \Vert_{L^\infty(\Omega)} \big)$. Next, we show that $y_0 \mapsto C(y_0)$ with $C(y_0)$ as in \eqref{Kornineqwithbadconstant} is a continuous mapping  on \linebreak $ \big( C^1(\overline{\Omega};\R^d), \Vert \cdot \Vert_{L^\infty(\Omega)} + \Vert\nabla (\cdot) \Vert_{L^\infty(\Omega)} \big)$. Indeed, if $\Vert \nabla \hat y_0 - \nabla \tilde y_0 \Vert_{L^\infty(\Omega)} < \eps$, we have
	\begin{align*}
	\vert \Vert \nabla y_1 - \nabla \tilde y_0 \Vert_{L^{\tilde p}(\Omega)} -  \Vert \nabla y_1 - \nabla \hat y_0 \Vert_{L^{\tilde p}(\Omega)} \vert \le \vert \Omega \vert^{1/{\tilde p}} \eps,
	\end{align*}
	as well as
	\begin{align*}
	&\vert \Vert (\nabla y_1)^T \nabla y_1 - (\nabla \tilde y_0)^T \nabla \tilde y_0 \Vert_{L^{\tilde p}(\Omega)} -  \Vert (\nabla y_1)^T \nabla y_1 - (\nabla \hat y_0)^T \nabla \hat y_0\Vert_{L^{\tilde p}(\Omega)} \vert  \\ &\leq  \Vert  (\nabla \tilde y_0)^T \nabla \tilde y_0 - (\nabla \hat y_0)^T \nabla \hat y_0  \Vert_{L^{\tilde p}(\Omega)}  \\
	 & =  \Vert  (\nabla \tilde y_0 - \nabla \hat y_0)^T \nabla \tilde y_0 - (\nabla \hat y_0)^T ( \nabla \hat y_0 - \nabla \tilde y_0 ) \Vert_{L^{\tilde p}(\Omega)}  \\
	 & \leq 2 { \REV \tilde C \END } \vert \Omega \vert^{1/{\tilde p}} \eps.
	\end{align*}
	Thus, Weierstraß' extremum principle implies that there exists a maximal constant $C(y_0^*)$ such that \eqref{Kornineqwithbadconstant} holds uniformly among functions in  $W_{{ \REV \tilde C \END }}$. In view of \eqref{aprioribounds}, we can conclude the proof.
	\end{proof}

	\begin{rem}[Nonlinear Korn's inequality]
		In the spirit of Lemma~\ref{nonlinearKornappl} and \cite[Theorem~3.3]{MielkeRoubicek},  \cite[Corollary 4.1]{Pompe} implies that there exists a constant $C>0$, only depending on $M>0$ and $\Omega$ such that  for any $y \in \mathscr{S}_{M}$ and $u \in W_0^{1,\tilde p}(\Omega;\R^d)$ it holds that	 
		   \begin{align}\label{Kornineqlinear}
		   \Vert \nabla u \Vert_{L^{\tilde p} (\Omega) } \leq C  \Vert (\nabla u)^T \nabla y + (\nabla y)^T \nabla u \Vert_{L^{\tilde p} (\Omega) }.
		   \end{align}
		  
	\end{rem}

    The   following lemma provides    properties    about the topology and lower semicontinuity. In addition, the lemma    will help    to prove lower    semicontinuity    of the slopes in    Lemma~\ref{lem:stronguppergradient3d} and Lemma~\ref{lem:repsmallstrain}.

\begin{lemma}[Properties of ($\mathscr{S}_{M}, \mathcal{D}$),  $\mathcal{D}$, and $\phi$]\label{th: metric space}
Let $M>0$. Then, we have
	\begin{itemize}
		\item[(i)] {\rm Completeness:} $(\mathscr{S}_{M}, \mathcal{D})$ is a complete metric space.
		\item[(ii)] {\rm Compactness:} If $(y_n)_n \subset \mathscr{S}_{M}$, then $(y_n)_n$ admits a subsequence converging weakly in \linebreak $W^{2,p}(\Omega;\R^d)$ and strongly in  $W^{1,\infty}(\Omega;\R^d)$.
		\item[(iii)] {\rm Topologies:} The topology on $\mathscr{S}_{M}$, induced by $\mathcal{D}$, coincides with the strong $W^{1,\infty}(\Omega;\R^d)$-topology, the strong $W^{1,\tilde p}(\Omega;\R^d)$-topology, and there exist $c, C>0$ depending on $M$ such that
\begin{align}\label{equivalence}
c \Vert \nabla y - \nabla \tilde y \Vert_{L^{\tilde p}(\Omega) } \leq \mathcal{D} (y, \tilde y) \leq C \Vert \nabla y - \nabla \tilde y \Vert_{L^{\tilde p}(\Omega) }
\end{align}
for all $y, \tilde y \in\mathscr{S}_{M}$.
		\item[(iv)] {\rm Lower semicontinuity of $\phi$:}    Let $(y_n)_n \subset \mathscr{S}_{M}$ be a sequence such that $\mathcal{D}(y_n,y) \to 0$ for some $y\in \mathscr{S}_{M}$. Then,   $\liminf_{n \to \infty} \phi(y_n) \ge \phi(y)$.   
		
		\item[(v)]  {\rm Continuity of ${\mathcal{D}}$:}  For sequences $(y_n)_n, \, (\tilde y_n)_n \subset \mathscr{S}_{M}$, with  $y_n \rightharpoonup y_0$ and $\tilde y_n \rightharpoonup \tilde y_0$ in $W^{2,p}(\Omega)$ we have
		
	$$\lim\limits_{n \to +\infty} \mathcal{D}(y_n,\tilde y_n) = {\mathcal{D}}  (y_0,\tilde y_0 ).$$ 
	 
	\end{itemize}

\end{lemma}
Before proving the lemma, we observe
that for any $F_0, F_1 \in \R^{d \times d}$ it holds that
\begin{align} \label{linearizethemetric}
& \quad \left(F_1-F_0\right)^{T} F_0+F_0^{T}\left(F_1-F_0\right)-\left(F_1^{T} F_1-F_0^{T} F_0\right)  =-\left(F_1^{T}-F_0^{T}\right)\left(F_1-F_0\right)  .
\end{align}
  \begin{proof}
   Concerning (i), the positivity follows by Lemma~\ref{nonlinearKornappl}, \ref{D_bound}, and Poincaré's inequality. The other conditions of a metric space follow directly from the specific choice of $\mathcal{D}$. \REV (ii) and \END the completeness of  $(\mathscr{S}_{M}, \mathcal{D})$ now follow as in \cite[Lemma 4.5]{MFMK}, while we prove (iii), \REV  (iv), and (v) \END in detail:      By Lemma~\ref{nonlinearKornappl} and \ref{D_bound}   we derive the first inequality  in \eqref{equivalence}. The second inequality follows from \ref{D_bound}, \eqref{aprioribounds}, and \eqref{linearizethemetric}. This particularly shows that the topology induced by $\mathcal{D}$ is equivalent to the $W^{1,\tilde p}(\Omega)$-topology. Let $(y_n)_n \subset \mathscr{S}_{M}$ be a sequence converging strongly in $W^{1,\tilde p}(\Omega;\R^d)$. Then, (ii) implies that the sequence also converges strongly in $W^{1,\infty}(\Omega;\R^d)$. Conversely, strong convergence in $W^{1,\infty}(\Omega;\R^d)$ yields strong convergence in $W^{1,\tilde p}(\Omega)$, concluding the proof of (iii).
   \REV We now prove (iv) and consider a sequence $(y_n)_n \subset \mathscr{S}_{M}$ such that $\mathcal{D}(y_n,y) \to 0$ for some $y\in \mathscr{S}_{M}$.  
   Using (ii), we can assume that $(y_n)_n $ converges weakly in $W^{2,p}(\Omega;\R^d)$ and strongly in $W^{1,\infty}(\Omega;\R^d)$. Thus, \ref{H_regularity} implies that the second term of $\phi$ as defined in \eqref{energyphieps} is lower semicontinuous. 
   Since the set $ \{ F \in \R^{d \times d} : \det (F) \geq {  \tilde C  }^{-1}, \ \vert F \vert \leq { \REV \tilde C  } \}$ is a compact subset of $GL^+(d)$, \eqref{aprioribounds}, \ref{W_regularity}, strong convergence in $W^{1,\infty}(\Omega;\R^d)$, and the dominated convergence theorem yield the convergence of the remaining parts of $\phi$. 
   Recalling \eqref{explicitchoiceofD}, one can proceed similarly in (v) as strong convergence in $W^{1,\infty}(\Omega;\R^d)$ implies that the integrand of $\mathcal{D}^{\tilde p}$ \END converges uniformly. 
\end{proof}
\REV Notice that the previous proof shows that no quasiconvexity of $W$ needs to be assumed due to the second gradient. \END

Proving existence of curves of maximal slopes \REV   requires the local slope to be lower semicontinuous. \END  Therefore, one usually resorts to specific convexity properties of \REV the energy and the distance. \END To prove these properties, we rely on the specific structure of $\mathcal{D}$ and recall that for any $\tilde p \in (1,+\infty)$ and $a\geq 0 , b \geq 0 $, we have
\begin{align}\label{inequalitytildep}
(a+b)^{\tilde p} \leq a^{\tilde p} + b^{\tilde p} +C( a^{\tilde p-1} b + b^{\tilde p-1} a ),
\end{align}
where the constant $C>0$ only depends on $\tilde p$, see e.g.\ \cite[Theorem~1]{Jameson}.
The following lemma specifies these convexity properties.
\begin{lemma}[Generalized convexity]\label{lem:localconv}
Let  $M>0$ and  $\tilde p \in \REVB (1,+\infty) \END $. 
\begin{itemize}

\item[(i)]   Then, there exists $\lambda<0$ and $\kappa>0$ such that for all $y_0$, $y_1 \in \mathscr{S}_{M}$ with $\Vert \nabla y_1 - \nabla y_0 \Vert_{L^\infty(\Omega)} \leq \kappa$ we have 
\begin{align}\label{conv:phi}
\phi (y_s) \leq (1-s) \phi (y_0) + s \phi(y_1) - \frac{1}{2} \lambda s (1-s) \REVB \Vert \nabla y_1 - \nabla y_0 \Vert_{L^2(\Omega)}^2 , \END
\end{align}
where $y_s \defas (1-s) y_0 + sy_1$. 

\item[(ii)] 
\REV Assume that \ref{Sadditional} holds. \END
%Suppose that \ref{W_regularity}--\ref{W_lower_bound} are complemented with the conditions $W(x,\Id) = 0$ and $W(x,F) \ge \ac \dist^2(F,SO(d))$ for all $F \in GL^+(d)$. Moreover, set $M = M' \delta^2$ for $M'>0$ and suppose that $\hat y = \id + \delta \hat u$ for $\hat u \in W^{2,p}(\Omega;\R^d)$ in \eqref{assumption:clampedboundary},  and  $f = \delta \tilde f$ for some $\tilde f \in L^\infty(\Omega)$.
 Then, there exists some $\delta'>0$ and some $\lambda>0$ such that for all $0<\delta<\delta'$  the functional $\phi$ is strictly convex on $\mathscr{S}_{M}$ and satisfies the estimate 
\begin{align}\label{conv:phi2}
\phi (y_s)  & \leq  (1-s) \phi (y_0) + s \phi(y_1) - \frac{1}{2} \lambda s (1-s) \REV \Vert \nabla y_1 - \nabla y_0 \Vert_{L^2(\Omega)}^2 , \END
\end{align} 
where $y_s \defas (1-s) y_0 + sy_1$.

\item[(iii)]  Then, there exists a constant $C>0$ such that for all $y_0$, $y_1 \in \mathscr{S}_{M}$  we have
\begin{align}\label{conv:D}
\mathcal{D} (y_s, y_0)^{\tilde p} \leq s^{\tilde p} \mathcal{D} (y_1,y_0)^{\tilde p} \big( 1+ C \Vert \nabla y_1 - \nabla y_0 \Vert_{L^\infty(\Omega) }^{\tilde p-1} + C \Vert \nabla y_1 - \nabla y_0 \Vert_{L^\infty(\Omega) } \big).
\end{align} 
 \end{itemize}

\end{lemma}
Notice that there might be better paths than convex combinations which allow for sharper estimates in \eqref{conv:D}. However, as \eqref{conv:phi} and \eqref{conv:phi2} rely on Taylor expansions, the use of convex combinations is convenient in our setting.
Moreover, we highlight that $\lambda$ in (ii) is positive. This is particularly relevant for proving decay estimates as the time $t$ tends to $+\infty$, see Section~\ref{sec:longtime}. 

%We stress that one needs to prove convexity of the distance in order to apply the prototypical theory from [AGS] which is not given in our setting. In particular, this would imply that there exist unique solutions to the continuous problem. As buckling phenomena could occur this cannot hold true.
\begin{proof}
%\emph{Step 1 ($\lambda$-convexity of $\phi$):}
 The proof of \eqref{conv:phi} relies on \cite[Proposition 3.2]{MielkeRoubicek}. The force term in \eqref{energyphieps} is linear and therefore convex. Due to the convexity of $H$, see \ref{H_regularity}, we only need to address the term $\int_\Omega W(x, \nabla y(x) ) \di x$.  Let $\kappa>0$ so small such that $\inf\limits_{x \in \bar \Omega} \det (\nabla y_s (x) ) \geq { \REV \tilde C \END }^{-1}/2$ for $y_0$, $y_1 \in \mathscr{S}_{M}$ with $\Vert \nabla y_1 - \nabla y_0 \Vert_{L^\infty(\Omega)} \leq \kappa$, where $y_s$ is defined in the statement and ${ \REV \tilde C \END }$ denotes the constant from \eqref{aprioribounds}. A Taylor expansion yields
\begin{align*}
\int_\Omega W(x, \nabla y_i(x) ) \di x &= \int_\Omega W(x, \nabla y_s(x) ) + \partial_F W(x, \nabla y_s(x) ) : (\nabla y_i - \nabla y_s) \di x \\
&\geq - C \Vert \nabla y_i - \nabla y_s \Vert_{L^2(\Omega) }^2
\end{align*}
for a constant $C>0$, only depending on $M$, as $\partial_{F^2}^2 W(x,F)$ attains its maximum on the set $\overline \Omega \times \{ F \in \REV GL^+(d) \END : \vert F \vert \leq { \REV \tilde C \END }, \ \det (F) \geq { \REV \tilde C \END }^{-1}/2 \}$ due to the regularity of $W$, see \ref{W_regularity}. We now multiply the previous inequality with $(1-s)$ for $i = 0$ and with $s$ for $i = 1$ and sum those \REV inequalities \END up, to get that
\begin{align*}
&\int_\Omega (1-s)W(x, \nabla y_0(x) ) +\int_\Omega s W(x, \nabla y_1(x) ) -  W(x, \nabla y_s(x) ) \\
&\geq  - C (1-s) \Vert \nabla y_0 - \nabla y_s \Vert_{L^2(\Omega) }^2 - C s \Vert \nabla y_1 - \nabla y_s \Vert_{L^2(\Omega) }^2.
\end{align*}
%Notice that we have $\Vert \nabla y_0 - \nabla y_1 \Vert_{L^\infty(\Omega)}^{2-\tilde p} \leq C$ for $C>0$ by \eqref{aprioribounds} and $\tilde p \in (1,2]$.  
%Thus, using Lemma~\ref{nonlinearKornappl}, \ref{D_bound}, and the expansion
\REVB Using the expansion \END
\begin{align*}
(1-s) \Vert \nabla y_0 - \nabla y_s \Vert_{L^2(\Omega) }^2 + s \Vert \nabla y_1 - \nabla y_s \Vert_{L^2(\Omega) }^2 = s(1-s) \Vert \nabla y_0 - \nabla y_1 \Vert_{L^2(\Omega) }^2
\end{align*}
%we derive
%\begin{align*}
%\int_\Omega W(x, \nabla y_s(x) ) \di x
%\leq & (1-s) \int_\Omega W(x, \nabla y_0(x) ) \di x  + s\int_\Omega  W(x, \nabla y_1(x) ) \di x \\
%& \qquad + C s(1-s) \mathcal{D} (y_0,y_1)^{\tilde p}.
%\end{align*}
\REVB and the choice \END $\lambda \defas -2C$ concludes the proof of \eqref{conv:phi} and (i).

For the proof of (ii) we can argue along the lines of \cite[Theorem~4.7(ii)]{MFMK} as the elastic energy satisfies identical assumptions. Thus, we have
 \begin{align*}
 \phi (y_s) &\leq (1-s) \phi (y_0) + s \phi(y_1) \\&\qquad - s(1-s) \int_\Omega \vert \sym(\nabla y_1 - \nabla y_0) \vert^2  \di x + \frac{1}{2} s(1-s) C \delta^{2/p} \int_\Omega \vert \nabla (y_1 - y_0) \vert^2 \di x.
 \end{align*}  
 As $y_1 - y_0 = 0 $ on $\Gamma_D$, we can employ Korn's inequality and \eqref{equivalence}, yielding the statement.
 
%  for $\tilde p  = 2$
%    and any $\delta < \delta'$, provided $\delta'>0$ sufficiently small. 
%   In the case $\tilde p \in (2, +\infty)$, we additionally have to use the fact that \eqref{aprioribounds} and \eqref{equivalence} yield
%  \begin{align*}
%   \int_\Omega \vert \nabla (y_1 - y_0) \vert^2 \di x & =  \int_\Omega \vert \nabla (y_1 - y_0) \vert^{\tilde p}  \vert \nabla (y_1 - y_0) \vert^{2 - \tilde p} \di x \\
%   &\geq   \int_\Omega \vert \nabla (y_1 - y_0) \vert^{\tilde p}  (2 C_0)^{2 - \tilde p} \di x \\
%   &\geq  (2 C_0)^{2 - \tilde p}     \mathcal{D}(y_1,y_0)^{\tilde p}  .
%  \end{align*}

%\emph{Step 2 (Convexity of $\mathcal{D}$):}  
We now address (iii).
 We exploit the explicit structure of the metric in \eqref{explicitchoiceofD} and \eqref{dissipationdistance}. As matrix multiplication is a linear operation and the metric satisfies the bound \ref{D_bound}, we can assume without loss of generality that $D (F_1,F_2) =  \vert F_1^T F_1 - F_2^T F_2 \vert$. An elementary expansion gives
\begin{align*}
& \quad (\nabla y_s)^T \nabla y_s - (\nabla y_0)^T \nabla y_0 \\& =((1-s) \nabla y_0 + s \nabla y_1)^T ((1-s) \nabla y_0 + s \nabla y_1) - (\nabla y_0)^T \nabla y_0 \\
& =s \Big(  - 2(\nabla y_0)^T \nabla y_0 +(\nabla y_0)^T \nabla y_1 \\ &\qquad + s(\nabla y_0)^T \nabla y_0 - s(\nabla y_0)^T \nabla y_1 + (\nabla y_1)^T \nabla y_0 - s (\nabla y_1)^T \nabla y_0 + s(\nabla y_1)^T \nabla y_1 \Big) \\
& = s \big( (\nabla y_1)^T \nabla y_1 - (\nabla y_0)^T \nabla y_0 + (s-1) (\nabla y_1 - \nabla y_0)^T (\nabla y_1 - \nabla y_0) \big).
\end{align*}
By using \eqref{inequalitytildep} we get that
\begin{align}
&\vert (\nabla y_s)^T \nabla y_s - (\nabla y_0)^T \nabla y_0 \vert^{\tilde p} \leq s^{\tilde p} \Big( \vert (\nabla y_1)^T \nabla y_1 - (\nabla y_0)^T \nabla y_0  \vert^{\tilde p} + \vert \nabla y_1 - \nabla y_0\vert^{2 \tilde p} \notag \\&\quad + C\big(\vert (\nabla y_1)^T \nabla y_1 - (\nabla y_0)^T \nabla y_0  \vert^{\tilde p-1}\vert \nabla y_1 - \nabla y_0\vert^{2} +\vert (\nabla y_1)^T \nabla y_1 - (\nabla y_0)^T \nabla y_0  \vert \vert \nabla y_1 - \nabla y_0\vert^{2(\tilde p-1)} \big) \Big) \notag \\
&\quad \eqcolon s^{\tilde p} ( A_1 + A_2 + A_3 + A_4), \label{expansion1}
\end{align}
 where each $A_i$, $i = 1,...,4$, corresponds to  exactly one summand in its respective order.
Recalling \eqref{dissipationdistance} and employing \eqref{Kornineq}, we find
\begin{align} \label{expansion2}
\int_\Omega A_2 \di x \leq \Vert \nabla y_1 - \nabla y_0\Vert_{L^\infty(\Omega)}^{\tilde p} \int_\Omega \vert \nabla y_1 - \nabla y_0 \vert^{\tilde p} \di x \leq C \Vert \nabla y_1 - \nabla y_0\Vert_{L^\infty(\Omega)}^{\tilde p} \mathcal{D} (y_1,y_0)^{\tilde p}.
\end{align}
Moreover, Hölder's inequality with coefficients $\tilde p/(\tilde p-1)$ and $\tilde p$ and Lemma~\ref{nonlinearKornappl} imply that
\begin{align}
\int_\Omega A_3 \di x & \leq C \Bigg(\int_\Omega \vert (\nabla y_1)^T \nabla y_1 - (\nabla y_0)^T \nabla y_0  \vert^{\tilde p} \di x \Bigg)^{(\tilde p-1)/\tilde p} \Bigg(\int_\Omega \vert \nabla y_1 - \nabla y_0\vert^{2 \tilde p} \di x \Bigg)^{1/\tilde p} \notag \\
& \leq  C\Vert \nabla y_1 - \nabla y_0\Vert_{L^\infty(\Omega)} \mathcal{D}(y_1,y_0)^{(\tilde p-1)}  \mathcal{D}(y_1,y_0)  \notag \\
& = C \Vert \nabla y_1 - \nabla y_0\Vert_{L^\infty(\Omega)} \mathcal{D}(y_1,y_0)^{\tilde p}.  \label{expansion3}
\end{align} 
Similarly, we derive by using Hölder's inequality with coefficients $\tilde p$ and $\tilde p/(\tilde p-1)$ and Lemma~\ref{nonlinearKornappl} that
\begin{align}
\int_\Omega A_4 \di x & \leq C \Bigg(\int_\Omega \vert (\nabla y_1)^T \nabla y_1 - (\nabla y_0)^T \nabla y_0  \vert^{\tilde p} \di x \Bigg)^{1/\tilde p} \Bigg(\int_\Omega \vert \nabla y_1 - \nabla y_0\vert^{2 \tilde p} \di x \Bigg)^{(\tilde p-1)/\tilde p} \notag \\
& \leq C \Vert \nabla y_1 - \nabla y_0\Vert_{L^\infty(\Omega)}^{\tilde p-1} \mathcal{D}(y_1,y_0)^{\tilde p}.  \label{expansion4}
\end{align}
 Thus, by using \eqref{aprioribounds} and \eqref{expansion1}--\eqref{expansion4} we can conclude the proof of \eqref{conv:D}.
\end{proof}

\begin{lemma}[Properties of $\vert \partial \phi \vert_{\mathcal{D}}$ for large strains]\label{lem:stronguppergradient3d}    Let $M>0$ and $\tilde p \in \REVB (1,+\infty) \END$. Then, there exists some $\kappa>0$ and $\lambda \in \R$ such that     the local slope $\vert \partial \phi \vert_{\mathcal{D}}$ 	\begin{itemize}
\item[(i)]admits the representation
\begin{align*}
\vert \partial \phi \vert_{\mathcal{D}}(y) = \sup\limits_{\substack{y\neq w\\\Vert y - w\Vert_{W^{   1,\infty}(\Omega)}    \le       \kappa    '}} \frac{ \left(\phi(y) - \phi(w)  + \REVB \tfrac{1}{2}    \lambda   \Vert \nabla y - \nabla w \Vert_{L^2(\Omega)}^2  \END   \right)^+}{\mathcal{D}(y,w) (1 + C \Vert \nabla w - \nabla y \Vert_{L^\infty(\Omega) }^{\tilde p-1} + C \Vert \nabla w - \nabla y \Vert_{L^\infty(\Omega) })^{1/\tilde p} }
\end{align*}
for any $\kappa' \leq \kappa$.
		\item[(ii)] is lower semicontinuous with respect to the weak topology in $W^{2,p}(\Omega;\R^d)$.
	\item[(iii)] is a strong upper gradient for $\phi$,
	\end{itemize} 	
 \end{lemma}

Similar properties of the local slope have been shown in \cite{MFMK,MFMKDimension,MFLMDimension3D1D}. We follow the lines of the latter articles, but highlight that the main difference to the latter articles is the representation in (i) which holds for \emph{large strains} and \REVB $\tilde p \in (1, +\infty)$. \END

\begin{proof}

We first prove the representation.  Whenever     $y \neq w$ and $  \phi(y) - \phi(w)  + \tfrac{1}{2}    \lambda   \REVB \Vert \nabla y - \nabla w \Vert_{L^2(\Omega)}^2    \END   >0$ we obtain by \eqref{conv:phi} 
\begin{align*}
	\frac{\phi(y) - \phi(y_s)}{\mathcal{D}(y,y_s)} \geq \frac{ \phi(y) - \phi(w)  + \tfrac{1}{2} \lambda (1-s)   \REVB \Vert \nabla y - \nabla w \Vert_{L^2(\Omega)}^2    \END  }{\mathcal{D}(y,w)}\frac{s\mathcal{D}(y,w)}{\mathcal{D}(y,y_s)}
\end{align*}
 for $ \Vert  y - w\Vert_{W^{1,\infty}(\Omega)} \leq       \kappa $ and $y_s = (1-s)y + sw$.
Then, \eqref{conv:D} implies that
%  With \eqref{conv:D},    for $s$ small enough,    this implies
\begin{align*}
\frac{\phi(y) - \phi(y_s)}{\mathcal{D}(y,y_s)} \geq \frac{ \phi(y) - \phi(w)  + \tfrac{1}{2} \lambda (1-s)   \REVB \Vert \nabla y - \nabla w \Vert_{L^2(\Omega)}^2    \END } {\mathcal{D}(y,w) (1 + C \Vert \nabla w - \nabla y \Vert_{L^\infty(\Omega) }^{\tilde p-1} + C \Vert \nabla w - \nabla y \Vert_{L^\infty(\Omega) })^{1/\tilde p} }   .
\end{align*}
Thus, by taking the limit    $s \to 0 $    and the supremum, we obtain for any $   \kappa    ' \leq        \kappa $ 
\begin{align}\label{slopelowerbound}
\vert \partial \phi \vert_{\mathcal{D}}(y) \geq \sup\limits_{\substack{y\neq w\\\Vert y - w\Vert_{W^{   1,\infty}(\Omega)}    \le       \kappa    '}} \frac{ \left(\phi(y) - \phi(w)  + \tfrac{1}{2}    \lambda        \REVB \Vert \nabla y - \nabla w \Vert_{L^2(\Omega)}^2    \END   \right)^+}{\mathcal{D}(y,w) (1 + C \Vert \nabla w - \nabla y \Vert_{L^\infty(\Omega) }^{\tilde p-1} + C \Vert \nabla w - \nabla y \Vert_{L^\infty(\Omega) })^{1/\tilde p} }.
\end{align}
We also    get    the reverse inequality
\begin{align}\label{slopeupperbound}
\vert \partial \phi \vert_{   \mathcal{D}}(y) =& \limsup\limits_{w \to y}  \frac{ \left(\phi(y) - \phi(w)  + \tfrac{1}{2}    \lambda         \REVB \Vert \nabla y - \nabla w \Vert_{L^2(\Omega)}^2    \END  \right)^+}{\mathcal{D}(y,w) (1 + C \Vert \nabla w - \nabla y \Vert_{L^\infty(\Omega) }^{\tilde p-1} + C \Vert \nabla w - \nabla y \Vert_{L^\infty(\Omega) })^{1/\tilde p} }  \nonumber \\
\leq& \sup\limits_{\substack{y\neq w\\\Vert y - w\Vert_{W^{   1,\infty}(\Omega)}    \le       \kappa    '}} \frac{ \left( \phi(y) - \phi(w)  + \tfrac{1}{2}    \lambda        \REVB \Vert \nabla y - \nabla w \Vert_{L^2(\Omega)}^2    \END  \right)^+}{\mathcal{D}(y,w) (1 + C \Vert \nabla w - \nabla y \Vert_{L^\infty(\Omega) }^{\tilde p-1} + C \Vert \nabla w - \nabla y \Vert_{L^\infty(\Omega) })^{1/\tilde p} }  ,
\end{align}
where we used    Lemma \ref{th: metric space}(v), $\tilde p >1$, and the fact    that $w \to y$ with respect to $\mathcal{D}$ implies $\Vert y - w \Vert_{W^{1,\infty}(\Omega)} \to 0$, see    Lemma~\ref{th: metric space}(ii),(iii). 

We are now ready to confirm
the lower semicontinuity. Consider a sequence $(y_n)_n$, $y_n \in \mathscr{S}_{M}$, converging weakly to $y \in \mathscr{S}_{M}$ in $W^{2,p}(\Omega)$ (and equivalently with respect to $\mathcal{D}$, see     Lemma~\ref{th: metric space}(ii),(iii)).   
For $w \neq y$ such that $\Vert w- y \Vert_{W^{   1,\infty}(\Omega)} \leq      \kappa/2    $,    we have    $w\neq y_n$  and  $\Vert w- y_n \Vert_{W^{   1,\infty}(\Omega)} \leq       \kappa $    for $n$ large enough,    and thus
\begin{align*}
	\liminf\limits_{n\to \infty} \vert \partial \phi\vert_{\mathcal{D}}(y_n)    & \ge     \liminf\limits_{n\to \infty} \frac{ \left( \phi(y_n) - \phi(w)  + \tfrac{1}{2}    \lambda     \REVB \Vert \nabla y_n - \nabla w \Vert_{L^2(\Omega)}^2    \END \right)^+}{\mathcal{D}(y_n,w) (1 + C \Vert \nabla w - \nabla y_n \Vert_{L^\infty(\Omega) }^{\tilde p-1} + C \Vert \nabla w - \nabla y_n \Vert_{L^\infty(\Omega) })^{1/\tilde p} } \\  
	&\geq \frac{ \left( \phi(y) - \phi(w)  + \tfrac{1}{2}    \lambda         \REVB \Vert \nabla y - \nabla w \Vert_{L^2(\Omega)}^2    \END   \right)^+}{\mathcal{D}(y,w) (1 + C \Vert \nabla w - \nabla y \Vert_{L^\infty(\Omega) }^{\tilde p-1} + C \Vert \nabla w - \nabla y \Vert_{L^\infty(\Omega) })^{1/\tilde p} },
\end{align*}
where we used    Lemma~\ref{th: metric space}(ii)--(v),    and    \eqref{slopelowerbound} for $   \kappa    ' =    \kappa$.     By taking the supremum with respect to $w$ and Lemma~\ref{lem:stronguppergradient3d}(i) for $   \kappa    ' =    \kappa/2$,    the
lower semicontinuity follows.

We now show (iii). As the local slope is a weak upper gradient in the sense of Definition \cite[Definition 1.2.2]{AGS} by \cite[Theorem  1.2.5]{AGS}, we only need to show that for an absolutely continuous curve $z\colon(a,b) \to \mathscr{S}_{M}$     (with respect to $\mathcal{D}$)  with $-\infty < a < b < +\infty$  satisfying $\vert \partial \phi\vert_{\mathcal{D}}(z) \vert z '\vert_{\mathcal{D}} \in L^1(a,b)$ the curve $\phi \circ z$ is absolutely continuous.   
We argue similarly to \cite[Corollary 2.4.10]{AGS} and extend the curve $z$ by continuity to $[a,b]$.

For any $t \in [a,b]$, we consider a ball $B^{W^{1,\infty}(\Omega;\R^d)}_{\kappa/2} (z(t))$, where $\kappa>0$ is the constant from (i).
Since the topologies induced by $W^{1,\infty}(\Omega;\R^d)$ and $\mathcal{D}$ coincide due to Lemma~\ref{th: metric space}(iii), the function $z$ is continuous with respect to the $W^{1,\infty}(\Omega;\R^d)$ topology. Thus, for each $t \in [a,b]$ there exists some $\delta_t >0$ such that $\vert t- \tilde a \vert < \delta_t$ implies that $z(\tilde a) \in B^{W^{1,\infty}(\Omega;\R^d)}_{\kappa/2} (z(t))$.
Since $[a,b]$ is compact there exists some $m \in \N$ such that $\bigcup_{i=1}^m B_{\delta_{t_i}}(t_i) = [a,b]$. 
Thus, there exists a partition $s_0<s_1<...<s_m$ such that $\Vert z(u) - z(v) \Vert_{W^{1,\infty}(\Omega;\R^d)} < \kappa$ for each $u,v \in [s_{i-1}, s_i]$.
We introduce the compact metric spaces $ \mathscr{S}^{i}_{M}\defas z([s_{i-1},s_i])$ with the metric $\mathcal{D}$, and we consider the related global slope $I^i_{\phi}$, see Definition~\cite[Definition 1.2.4]{AGS}. Note that \REVB \eqref{aprioribounds} and \END \eqref{equivalence} give $\sup\limits_{x,y \in \mathscr{S}^{i}_{M}} \mathcal{D}(x,y) \leq C$ \REVB and $\Vert \nabla y - \nabla w \Vert_{L^2(\Omega)}^2 \leq C {\mathcal{D}(y,w)}^{\tilde p \wedge 2}$, where $\wedge$ denotes the minimum function. \END Hence, \eqref{aprioribounds}  implies that
\begin{align*}
I^i_{\phi}(y) &= \sup\limits_{ {y\neq w\in \mathscr{S}^{i}_{M} }} \frac{ \left(\phi(y) - \phi(w)     \right)^+}{\mathcal{D}(y,w)   }  \\
&\leq C \sup\limits_{\substack{y\neq w\in \mathscr{S}^{i}_{M} \\ \Vert y - w\Vert_{W^{   1,\infty}(\Omega)}    \le       \kappa}} \frac{ \left(\phi(y) - \phi(w)    +\tfrac{1}{2}    \lambda     \REVB \Vert \nabla y - \nabla w \Vert_{L^2(\Omega)}^2    \END   \right)^+ }{\mathcal{D}(y,w)  (1 + C \Vert \nabla w - \nabla y \Vert_{L^\infty(\Omega) }^{\tilde p-1} + C \Vert \nabla w - \nabla y \Vert_{L^\infty(\Omega) })^{1/\tilde p} } \\
& \qquad -  \sup\limits_{\substack{y\neq w\in \mathscr{S}^{i}_{M} \\ \Vert y - w\Vert_{W^{   1,\infty}(\Omega)}    \le       \kappa}}   \frac{  \REVB \tfrac{1}{2}    \lambda  \Vert \nabla y - \nabla w \Vert_{L^2(\Omega)}^2    \END }{     \mathcal{D}(y,w) } \\
&\leq C \sup\limits_{\substack{y\neq w\in \mathscr{S}_{M} \\ \Vert y - w\Vert_{W^{   1,\infty}(\Omega)}    \le       \kappa}} \frac{ \left(\phi(y) - \phi(w)    +\tfrac{1}{2}    \lambda       \REVB \Vert \nabla y - \nabla w \Vert_{L^2(\Omega)}^2    \END    \right)^+ }{\mathcal{D}(y,w)  (1 + C \Vert \nabla w - \nabla y \Vert_{L^\infty(\Omega) }^{\tilde p-1} + C \Vert \nabla w - \nabla y \Vert_{L^\infty(\Omega) })^{1/\tilde p} } + C \\
& = C \vert \partial \phi \vert_{\mathcal{D}}(y) + C.
\end{align*}
We now choose for any $u<v \in [a,b]$ the indices $i_u$ and $i_v \in \{1,...,m\}$ such that $u \in [s_{i_u-1},s_{i_u}]$ and $v \in [s_{i_v-1},s_{i_v}]$.
Since the global slope is a strong upper gradient due to \cite[Theorem 1.2.5]{AGS}, we have by the definition of a strong upper gradient that
\begin{align*}
  \vert \phi(z(u)) - \phi(z(v))\vert &\leq \vert \phi(z(u)) - \phi(z(s_{i_u}))\vert \\& \qquad + \sum_{i = i_{u}+1}^{i_{v}-1}  \vert \phi(z(s_i)) - \phi(z(s_{i-1}))\vert +\vert \phi(z(v)) - \phi(z(s_{i_v-1}))\vert \\
&\leq \int_u^{s_{i_u}} I^{i_u}_{\phi}(z(r)) \vert z' \vert_{\mathcal{D}}(r)\di r + \sum_{i = i_{u}+1}^{i_{v}-1} \int_{s_{i-1}}^{s_{i}} I^{i}_{\phi}(z(r)) \vert z' \vert_{\mathcal{D}}(r)\di r  \\&\qquad + \int_{s_{i_v-1}}^{v} I^{i_v}_{\phi}(z(r)) \vert z' \vert_{\mathcal{D}}(r)\di r \\
&\leq \int_u^v  ( C \vert \partial \phi \vert_{\mathcal{D}}(z(r)) + C) \vert z' \vert_{\mathcal{D}}(r)\di r.
\end{align*}
The right-hand side is finite due to our assumption and \cite[Theorem 1.1.2]{AGS}. Thus, $\phi(z(\cdot))$ is absolutely continuous and (iii) follows.
% Eventually, we observe that the proof in (i) also applies when  \eqref{conv:phi2} is substituted by
% \eqref{conv:phi3}, proving the representation in (i) also with $ \mathcal{D}(y,w)^{\tilde p} $   replaced by $ \mathcal{D}(y,w)^{2}    $.
\end{proof}
\REVB A refined representation can be derived for small data, where $\lambda>0$ can be chosen to be positive, and the local supremum can be replaced by a global one. \END
\begin{lemma}[Representation of $\vert \partial \phi \vert_{\mathcal{D}}$ for small strains]\label{lem:repsmallstrain}    Let $M>0$ and $\tilde p \in (1,+\infty)$.
	% We
	% suppose that \ref{W_regularity}--\ref{W_lower_bound} are complemented with the conditions $W(x,\Id) = 0$ and $W(x,F) \ge \ac \dist^2(F,SO(d))$ for all $F \in GL^+(d)$.
	 Moreover, \REV assume that \ref{Sadditional} holds and \END set $M = M' \delta^2$ for $M'>0$.
	 %and suppose that $\hat y = \id + \delta \hat u$ for $\hat u \in W^{2,p}(\Omega;\R^d)$ in \eqref{assumption:clampedboundary},  and  $f = \delta \tilde f$ for some $\tilde f \in L^\infty(\Omega)$.
	Then, there exists some $\delta' >0 $ such that for all $0<\delta < \delta'$      the local slope $\vert \partial \phi \vert_{\mathcal{D}}$ 	 
	 admits the representation
	\begin{align*}
	\vert \partial \phi \vert_{\mathcal{D}}(y) = \sup\limits_{y\neq w} \frac{ \left(\phi(y) - \phi(w)  + \tfrac{1}{2}    \lambda  \REV   \Vert \nabla y - \nabla w \Vert_{L^2(\Omega)}^2 \END    \right)^+}{\mathcal{D}(y,w) (1 + C \Vert \nabla w - \nabla y \Vert_{L^\infty(\Omega) }^{\tilde p-1} + C \Vert \nabla w - \nabla y \Vert_{L^\infty(\Omega) })^{1/\tilde p} }
	\end{align*}
	for any $\lambda >0$. 	
		%Moreover, we can replace $ \mathcal{D}(y,w)^{\tilde p} $   in (i) by $\int_\Omega \vert \nabla w - \nabla y \vert^{2}  \di x  $
	\end{lemma}

	\begin{proof}
We can argue along the lines of Lemma~\ref{lem:stronguppergradient3d} and have to replace \eqref{conv:phi} by
\eqref{conv:phi2}. Moreover, one uses \cite[Lemma~4.2]{MFMK} to observe that $\Vert \nabla w_1  -\nabla w_0 \Vert_{L^\infty(\Omega)} \leq C \delta^{2/p}$. Thus, one needs to choose $\delta'$ so small such that $C (\delta')^{2/p} \leq \kappa$.
	\end{proof}

  One of our goals will be to establish a relation of curves of maximal slope to the system of equations. The natural idea is to make use of the energy-dissipation-balance \eqref{maximalslope}. For this purpose, we use a finer representation of the local slope.
For convenience, we introduce a differential operator associated to the perturbation $P$ and define
\begin{align*}
(\mathcal{L}_P(x,\nabla^2 y) )_{ij} \defas - \big( \diver ( \partial_G P(x,\nabla^2 y) ) \big)_{ij} = - \sum\limits_{k=1}^d \partial_k (\partial_G P(x,\nabla^2 y) )_{ijk}, \qquad i,j \in \{1,...,d\}.
\end{align*}
Given a smooth function $\varphi\colon \Omega \to \R^d$, we use the notation $(\nabla \varphi)_{ik} = \partial_k \varphi_i$ and $(\nabla^2 \varphi)_{ijk} = \partial^2_{jk} \varphi_i$ for $i,j,k \in \{1,..., d\}$. Then, for later purposes, we note that this operator particularly satisfies
\begin{align}\label{identitytensor}
\int_\Omega \mathcal{L}_P(x,\nabla^2 y)   : \nabla \varphi  \di x   =  \int_\Omega \partial_G P(x,\nabla^2 y)   \cdddot  \nabla^2 \varphi \di x,
\end{align}  whenever integration by parts makes sense and $\varphi$ vanishes on $\partial \Omega$. 

%\begin{align*}
%	\int_\Omega \partial_G P(x,\nabla^2 y)   \cdddot  \nabla^2 \varphi \di x = - \int_\Omega \diver ( \partial_G P(x,\nabla^2 y) ) : \nabla \varphi \di x + \int_{\Gamma_N} \partial_G P(x,\nabla^2 y) : \nabla \varphi \otimes \nu \di \mathcal{H}^{d-1}  
%\end{align*}
%for all $\varphi \in W^{1,\tilde p}_{\Gamma_D} (\Omega;\R^d) $.

%Consider the fourth order tensor $H_{Y} \colon Y^{-T} \R^{d \times d}_{\rm sym} \to Y \R^{d \times d}_{\rm sym} $ defined as
%$H_{Y} Z \defas 2 Y  \big( Z^T Y + Y^T Z \big)$. The inverse is given by $H_Y^{-1} B \defas \frac{1}{2} Y^{-T} (B^T Y^{-T} + Y^{-1} B )$
%By its definition the tensor is bijective.
%, i.e., the square root $\sqrt{H_Y}$ and the inverse $\sqrt{H_Y}^{-1}$ exist.

\begin{lemma}[Finer representation of the slope]\label{finerepresentation}
	\REV Let $M>0$ and $\tilde p \in (1,+\infty)$. \END For any $ y \in \mathscr{S}_{M}$ the local slope $\vert \partial \phi \vert_{\mathcal{D}}$ admits the representation
\begin{align}\label{finerepresentationformula}
\vert \partial \phi \vert_{\mathcal{D}}(y) = \begin{cases}
 \left( \int_\Omega  {\tilde p} R(x,\nabla  y, \nabla \bar w)   \di x\right)^{1-1/\tilde p}, &\quad \diver \mathcal{L}_ P(\nabla^2 y) \in W^{-1,{\tilde p/ (\tilde p -1 )}}(\Omega), \\
+\infty, &\quad {\rm{else}},
\end{cases}
\end{align}
where $\bar w \in W_0^{1,\tilde p}(\Omega) $ satisfies
\begin{align}\label{eulerlagrangefordifferentialop}
&\int_\Omega \big( \partial_F W(x,\nabla y) + \mathcal{L}_P(x,\nabla^2 y) \big) : \nabla \varphi - f  \cdot \varphi \di x = \int_\Omega \partial_{\dot{F}}R(x,\nabla y, \nabla \bar w) : \nabla \varphi \di x
% -  \int_{\Gamma_N} g_\eps  \cdot \varphi  \di \mathcal{H}^{d-1} \notag\\ &\qquad 
\end{align}
 for any $\varphi \in W_0^{1,\tilde p}(\Omega)$. 
\end{lemma}
Notice that the representation is different compared to the representation in \cite[Lemma~5.5]{MFMK} where the case $\tilde p = 2$ is covered. The idea is similar, but for the relation of local slope with \eqref{weaksolutiondef} in the proof of Theorem~\ref{maintheorem1}, it is convenient to keep $\bar w$ in the representation of the slope.

Before proving the lemma, we state identities satisfied by the viscous stress tensor. To this end, notice that the scalar product of matrices $U,V \in \R^{d \times d}$ satisfies $U:V = U^T : V^T$. 
Recalling  \eqref{explicitex},
the symmetry of the scalar product implies that for any $F_0, F_1, F_2 \in \R^{d \times d}$ we have
\begin{align}\label{explicit2}
 &\qquad\partial_{\dot{F}}R(x,F_0, F_1) : F_2  \notag \\
 &= \vert A(x) (F_1^T F_0 + F_0^T F_1) \vert^{\tilde p - 2}   \big( F_0 (F_1^T F_0 + F_0^T F_1) A(x)^T A(x)  +   F_0 A(x)^T A(x) (F_1^T F_0 + F_0^T F_1) \big)  : F_2 \notag \\
  &= \vert A(x) (F_1^T F_0 + F_0^T F_1) \vert^{\tilde p - 2}   \big( A(x) (F_1^T F_0 + F_0^T F_1)    : A(x) ( F_2^T F_0 +F_0^T F_2 ) \big).
\end{align}
Hence, in the case $F_1 = F_2$, we thus find by \eqref{defRx} that
\begin{align}\label{relationRandderivative}
 \partial_{\dot{F}}R(x,F_0, F_1) : F_1  = \vert A(x) (F_1^T F_0 + F_0^T F_1) \vert^{\tilde p}     = {\tilde p} R(x,F_0, F_1).
\end{align}

\begin{proof}
For the sake of simplicity, we drop the $x$-dependence of the densities $W$, $R$ and $P$.
Assume first that $\diver \mathcal{L}_P(\nabla^2 y) \in W^{-1,{\tilde p/ (\tilde p -1 )}}(\Omega)$. We consider the following minimization problem
\begin{align}\label{minimization problem}
\min\limits_{w \in W^{1,{\tilde p}}_0(\Omega)} \int_\Omega R(\nabla y, \nabla w) - \big( \partial_F W(\nabla y) + \mathcal{L}_P(\nabla^2 y) \big) : \nabla w + f  \cdot w   \di x .
\end{align}
Recall that $y$ satisfies the bounds \eqref{aprioribounds} and that $f \in L^\infty(\Omega)$. Thus, 
compactness follows by \eqref{Kornineqlinear}, \eqref{defRx}, and \ref{D_bound}. Sequential lower semi-continuity in $W^{1,{\tilde p}}(\Omega)$ follows from the convexity of the map $\dot F \mapsto R(x,F,\dot F)$.
Since the norm is strictly convex, the unique solution $\bar w \in W_0^{1,{\tilde p}}(\Omega) $ exists and satisfies \eqref{eulerlagrangefordifferentialop}.

In what follows, we neglect the force term. As this term corresponds to a linear perturbation, one needs to perform minor adaptions in this case.
 
\emph{Step 1 (Lower bound):} Assume again that $\diver \mathcal{L}_P(\nabla^2 y) \in W^{-1,{\tilde p/ (\tilde p -1 )}}(\Omega)$. To derive the lower bound for the local slope, we wish to compute the local slope for the point $y$ along trajectories $s \mapsto y - s \bar w$ as long as $\bar w \neq 0 $. Since $\bar w$ does not enjoy sufficient regularity properties, we introduce for any $\gamma>0$ some $w_\gamma \in C_c^\infty(\Omega;\R^d)$ such that $\Vert \bar w - w_\gamma \Vert_{W^{1,{\tilde p}}(\Omega)} \leq \gamma$ and consider $w_s \defas y - s w_\gamma$.  Without loss of generality one can find some $s'>0$ such that for all $s < s'$ we have $w_s \in \mathscr{S}_{M}$. (Otherwise, we consider the space $\mathscr{S}_{\eps,M'}$ for $M' >M$.)   Thus, a Taylor expansion and convexity of $P$ imply that
\begin{align*}
\phi(w_s) - \phi(y) &\geq \int_\Omega \partial_F W(\nabla y): (\nabla w_s - \nabla y ) \di x \\ &\quad - C \Vert \nabla w_s - \nabla y \Vert_{L^2(\Omega)}^2  + \int_\Omega \partial_G P(\nabla^2 y): (\nabla^2 w_s - \nabla^2 y ) \di x\\
 &= s \int_\Omega \partial_F W(\nabla y): \nabla w_\gamma  +\partial_G P(\nabla^2 y): \nabla^2 w_\gamma \di x + O(s^2) 
\end{align*}
On the other hand, \eqref{inequalitytildep} yields
\begin{align*}
 \mathcal{D}(y,w_s)^{\tilde p} 
&=\int_\Omega  \vert A(x) (\nabla y)^T \nabla y - A(x)(\nabla  y - s \nabla w_\gamma)^T (\nabla  y - s \nabla w_\gamma)  \vert^{\tilde p} \di x \\
&=\int_\Omega  \vert s A(x)\left( (\nabla  y)^T \nabla w_\gamma + (\nabla w_\gamma)^T \nabla y \right) - s^2 A(x) (\nabla w_\gamma)^T  \nabla w_\gamma \vert^{\tilde p} \di x \\
& \leq s^{\tilde p} \int_\Omega  \vert A(x) (\nabla  y)^T \nabla w_\gamma + A(x)(\nabla w_\gamma)^T \nabla y  \vert^{\tilde p} \di x + O(s^{2{\tilde p}-1 }+ s^{\tilde p +1}).
\end{align*}
Since $\mathcal{D}(y,w_s) \to 0$ as $s \to 0$, by integration by parts we derive 
\begin{align*}
\vert \partial \phi \vert_{\mathcal{D}}(y) \geq \limsup\limits_{s \to 0} \frac{\left(\phi(y) - \phi(w_s) \right)^+}{\mathcal{D}(y,w_s)} \geq \frac{\left( \int_\Omega \partial_F W(\nabla y): \nabla w_\gamma  + \mathcal{L}_P(\nabla^2 y): \nabla w_\gamma \di x \right)^+ }{\left( \int_\Omega  \vert A(x) (\nabla  y)^T \nabla w_\gamma + A(x) (\nabla w_\gamma)^T \nabla y  \vert^{\tilde p} \di x\right)^{1/{\tilde p}}} \eqcolon \Phi(w_\gamma),
\end{align*}
where $\Phi$ is a function defined for every   $0 \neq w \in W_0^{1,{\tilde p}}(\Omega)$.  
Thus, we derive by using  \eqref{relationRandderivative} and \eqref{eulerlagrangefordifferentialop}
\begin{align*}
\Phi(\bar w) - \Phi(w_\gamma) + \vert \partial \phi \vert_{\mathcal{D}}(y) \geq \Phi(\bar w) &=  \frac{\left( \int_\Omega \partial_F W(\nabla y): \nabla \bar w + \mathcal{L}_P(\nabla^2 y): \nabla \bar w \di x \right)^+ }{\left( \int_\Omega  \vert  A(x)(\nabla  y)^T \nabla \bar w+ A(x)(\nabla \bar w)^T \nabla y  \vert^{\tilde p} \di x\right)^{1/{\tilde p}}}\\
&= \left( \int_\Omega  \vert  A(x)(\nabla  y)^T \nabla \bar w+ A(x)(\nabla \bar w)^T \nabla y  \vert^{\tilde p} \di x\right)^{1-1/{\tilde p}}.
\end{align*}
 As $\Phi(\bar w ) - \Phi(w_\gamma) \to 0$ for $\gamma \to 0$, we find
\begin{align}\label{zwischenschrittlower}
\vert \partial \phi \vert_{\mathcal{D}}(y) \geq \left( \int_\Omega  \vert  A(x)(\nabla  y)^T \nabla \bar w+ A(x)(\nabla \bar w)^T \nabla y  \vert^{\tilde p} \di x\right)^{1-1/{\tilde p}},
\end{align}
which is the desired lower bound due to \eqref{relationRandderivative}.
According to the definition of the local slope we have $\vert \partial \phi \vert_{\mathcal{D}}(y) \geq 0$ and thus the lower bound also makes sense if $\bar w = 0$.

Now assume that $ \diver \mathcal{L}_P(\nabla^2 y)  \notin W^{-1,{\tilde p/ (\tilde p -1 )}}(\Omega)$.
Let $(y_n)_n$ be a sequence of smooth functions converging strongly to $y$ in $W^{2,p}(\Omega)$.
Let $\bar w_n \in W^{1,{\tilde p}}_0(\Omega)$ be the solutions 
to \eqref{minimization problem} corresponding to $y_n$.
We first show that $(\bar w_n)_n$ is not bounded in $L^{\tilde p }(\Omega)$.
 Indeed, otherwise
 we get by  \eqref{identitytensor}, \ref{H_bounds}
\eqref{eulerlagrangefordifferentialop},   \eqref{explicitex},  \eqref{aprioribounds}, \ref{W_regularity} and Hölder's inequality that
\begin{align*}
\left\vert \int_\Omega \mathcal{L}_P(\nabla^2 y)   : \nabla \varphi  \di x \right\vert &= \left\vert \int_\Omega \partial_G P(\nabla^2 y)   : \nabla^2 \varphi \di x \right\vert = \lim\limits_{n \to \infty}  \left\vert \int_\Omega \partial_G P(\nabla^2 y_n)   : \nabla^2 \varphi \di x \right\vert \\
 &=\lim\limits_{n \to \infty}  \left\vert \int_\Omega ( \partial_{\dot{F}}R(x,\nabla y_n, \nabla \bar w_n)  - \partial_F W(\nabla y_n)  ) : \nabla \varphi \di x  \right\vert \leq C \Vert \nabla \varphi \Vert_{L^{\tilde p }(\Omega) }  .
\end{align*}
for all $\varphi \in W^{2,p}_0(\Omega)$. This, however, contradicts $\diver \mathcal{L}_P(\nabla^2 y)   \notin W^{-1,{\tilde p/ (\tilde p -1 )}}(\Omega)$. 

Fix $\eps>0$.
Due to Lemma~\ref{lem:stronguppergradient3d}(i) we find $\kappa>0$ such that
for every $y_n$ there exists $w_n \neq y$ with  $  \Vert w_n - y_n \Vert_{L^\infty(\Omega)} < \kappa/2 $ such that
\begin{align*}
\vert \partial \phi \vert_{\mathcal{D}}(y_n) - \eps \leq   \frac{ \left(\phi(y_n) - \phi(w_n)  + \tfrac{1}{2}    \lambda     \mathcal{D}(y_n,w_n)^{\tilde p}    \right)^+}{\mathcal{D}(y_n,w_n) (1 + C \Vert \nabla w_n - \nabla y_n \Vert_{L^\infty(\Omega) }^{\tilde p-1} + C \Vert \nabla w_n - \nabla y_n \Vert_{L^\infty(\Omega) })^{1/\tilde p} }.
\end{align*}
Since $\phi$ and $\mathcal{D}$ are continuous with respect to strong convergence in $W^{2,p}(\Omega)$, Lemma~\ref{lem:stronguppergradient3d}(i) implies for $n \in \N$ large enough that
\begin{align}
\vert \partial \phi \vert_{\mathcal{D}}(y_n) - \eps \leq   \vert \partial \phi \vert_{\mathcal{D}}(y) + \eps. \label{help1}
\end{align}
Since $y_n$ is smooth we particularly derive the bound \eqref{zwischenschrittlower} for $y_n$ with $\bar w$ replaced by $\bar w_n$. Since $(\bar w_n)_n$ is not bounded in $L^{\tilde p}(\Omega)$, the left-hand side of \eqref{help1} tends to $+\infty$, due to \eqref{Kornineqlinear}. This concludes the proof in the case $ \diver \mathcal{L}_P(\nabla^2 y)  \notin W^{-1,{\tilde p/ (\tilde p -1 )}}(\Omega)$.

\emph{Step 2 (Upper bound):} We now show the reverse inequality. 
Assume again that $ \diver \mathcal{L}_P(\nabla^2 y)  \in W^{-1,{\tilde p/ (\tilde p -1 )}}(\Omega)$.
Given $F_0,F_1 \in \R^{d \times d}$, \eqref{linearizethemetric}, \eqref{inequalitytildep}, the triangle inequality, and \ref{D_bound} imply that
\begin{align*}
 \vert A(x) \left(F_1-F_0\right)^{T} F_0 &+ A(x) F_0^{T}\left(F_1-F_0\right) \vert^{\tilde p}  \leq \left( \vert A(x) F_1^{T} F_1- A(x) F_0^{T} F_0\vert + \vert F_1-F_0 \vert^2 \right)^{\tilde p}   \\
 & \leq  \vert A(x) F_1^{T} F_1- A(x) F_0^{T} F_0\vert^{\tilde p} +  \vert F_1-F_0  \vert^{2\tilde p} \\
 & \qquad + C \vert F_1^{T} F_1-  F_0^{T} F_0\vert^{\tilde p-1} \vert F_1-F_0 \vert^2 + C \vert F_1^{T} F_1-  F_0^{T} F_0\vert \vert F_1-F_0 \vert^{2 (\tilde p -1)}\\
 & \leq  \vert A(x) F_1^{T} F_1- A(x) F_0^{T} F_0\vert^{\tilde p} +  \vert F_1-F_0  \vert^{2\tilde p} \\
 & \qquad + C ( \vert F_1 \vert + \vert F_0 \vert )^{\tilde p -1} \vert F_1-F_0 \vert^{\tilde p +1} + C ( \vert F_1 \vert + \vert F_0 \vert ) \vert F_1-F_0 \vert^{2 \tilde p -1} .
	\end{align*}
Using the previous inequalities for $F_1 = \nabla w$ and $F_0 = \nabla y$, by \eqref{aprioribounds} we find that
\begin{align}\label{helpsbelow}
&\int_\Omega  \vert  A(x)(\nabla  y)^T (\nabla y - \nabla w) +  A(x)(\nabla y - \nabla w) ^T \nabla y  \vert^{\tilde p} \di x \leq  \mathcal{D}(y,w)^{\tilde p}  +  C \int_\Omega \vert \nabla w - \nabla y \vert^{  {\tilde p}+1 \wedge {2\tilde p}-1 }  \di x ,
\end{align}
where $\wedge$ denotes the minimum function.
  A Taylor expansion  and the convexity of $P$ yield
\begin{align*}
& \quad \ \vert \partial \phi \vert_{\mathcal{D}}(y)  \leq \limsup\limits_{w \to y} \frac{\left(\phi(y) - \phi(w) \right)^+}{\mathcal{D}(y,w)} \\
& \leq \limsup\limits_{w \to y} \frac{\left(\int_\Omega \partial_F W(\nabla y): ( \nabla y - \nabla w) + \partial_G P(\nabla^2 y) : (\nabla^2 y - \nabla^2 w ) \di x\right)^+  }{ \left(  \int_\Omega  \vert A(x)  (\nabla  y)^T (\nabla y - \nabla w) + A(x) (\nabla y - \nabla w) ^T \nabla y  \vert^{\tilde p} \di x -  C \int_\Omega \vert \nabla w - \nabla y \vert^{  {\tilde p}+1 \wedge {2\tilde p}-1 }  \di x \right)^{1/{\tilde p}}  } \\
& \qquad + \limsup\limits_{w \to y}    \frac{  C \Vert \nabla y - \nabla w \Vert_{L^\infty(\Omega) } \int_\Omega \vert \nabla y - \nabla w \vert \di x}{ \mathcal{D}(y,w)} .
\end{align*}
 Lemma~\ref{th: metric space}(ii) and (iii) imply that the second term vanishes.
%Frame-indifference allows to employ \cite[Lemma~4.5]{dimensionsreduktion}, yielding 
%\begin{align*}
% \partial_F W(F)   : \dot F &= \frac{1}{2} F^{-1}  \partial_F W  ( F ) : (F^T \dot F + \dot F^T F),
%\end{align*} 
%XXX do it for second order XXX
%$\partial_F W(F) + \mathcal{L}_P(\nabla^2 y)$
%U
%
%Eventually, Hölder's inequality  with coefficients $q $ and $q/(q-1)$ implies that
In view of \eqref{eulerlagrangefordifferentialop}, the first term can be estimated by using \eqref{identitytensor}, Hölder's inequality, and \eqref{explicit2}, i.e.,
\begin{align*}
&\int_\Omega \partial_F W(\nabla y): ( \nabla y - \nabla w) + \partial_G P(\nabla^2 y) : (\nabla^2 y - \nabla^2 w ) \di x = \int_\Omega \partial_{\dot{F}}R(x,\nabla y, \nabla \bar w) : (\nabla y - \nabla w ) \di x \\
&\qquad\leq \left( \int_\Omega  \vert A(x) (\nabla  y)^T (  \nabla \bar w) + A(x) (\nabla \bar w) ^T \nabla y  \vert^{\tilde p} \di x \right)^{(\tilde p-1) /\tilde p} \\ &\qquad \qquad \qquad \qquad \qquad \cdot \left( \int_\Omega  \vert A(x) (\nabla  y)^T (\nabla y - \nabla w) +  A(x)(\nabla y - \nabla w) ^T \nabla y  \vert^{\tilde p} \di x  \right)^{1/\tilde p}.
\end{align*}
As $\tilde p \in (1,+\infty)$ yields $ \frac{{\tilde p}+1 \wedge {2\tilde p}-1 }{\tilde p} > 1$, we can use \eqref{Kornineqlinear} and \ref{D_bound} for the bound of the remainder. By taking the limit, this shows the upper bound in the case $\tilde w \neq 0$.
The right-hand side is equal to $0$ if $\bar w = 0$.
Concluding, we derive the upper bound
\begin{align}\label{zwischenschrittupper}
\vert \partial \phi \vert_{\mathcal{D}}(y) \leq  \left( \int_\Omega  \vert  A(x)(\nabla  y)^T \nabla \bar w+ A(x)(\nabla \bar w)^T \nabla y  \vert^{\tilde p} \di x\right)^{1-1/{\tilde p}}.
\end{align}
Recalling \eqref{relationRandderivative}, this concludes the proof.
\end{proof}

In the last part of this section, we address the existence of weak solutions and prove Theorem~\ref{maintheorem1}.

\begin{proof}[Proof of Theorem~\ref{maintheorem1}]

We divide the proof into four steps. First, we prove \REV (i), \END i.e., the existence of a curve of maximal slope $y \colon [0,\infty) \to \mathscr{S}_{M}$ for $\phi$ with respect to $\vert \partial \phi \vert_{\mathcal{D}}$. Then, we  prove the regularity stated in \REV Definition~\ref{def:weakformulation} (Step 2). \END Step 3 consists in deriving sharp estimates for $\vert y' \vert_{\mathcal{D}}$ and $\frac{{\rm d}}{{\rm d} t}\phi \circ y$. This allows to relate the curve to the system of equations in Step 4.

\noindent \textit{Step 1:}
We first prove (i). Our goal is to employ Theorem~\ref{th:abstract convergence 2} for the metric space $(\mathscr{S}_{M}, \mathcal{D})$, elastic energy $\phi$ and the weak $W^{2,p}(\Omega)$-topology, denoted by $\sigma$ in this result. Let $(\tau_k)_k$ be a sequence of time-steps such that $\tau_k \to 0$ as $k \to \infty$.
 The existence of time-discrete solutions, as in \eqref{time-discretescheme},
 is guaranteed by the direct method of the Calculus of Variations.  Indeed, let $Y_{\tau_k}^{n-1}$, $n \geq 1$ be given. Then, the coercivity of $\mathbf{\Phi}(\tau_k, Y_{\tau_k}^{n-1},\cdot)$ with respect to the weak $W^{2,p}(\Omega;\R^d)$-topology is a consequence of \eqref{aprioribounds}. Lower semicontinuity follows from Lemma~\ref{th: metric space}(iv) and (v).
 Lemma~\ref{th: metric space}(i) ensures that $(\mathscr{S}_{M}, \mathcal{D})$ is a complete metric space. Moreover, \eqref{compatibility}--\eqref{eq: implication} follow by
 Lemma~\ref{th: metric space}(ii),(iv),(v) and
  Lemma~\ref{lem:stronguppergradient3d}(ii).
Due to Lemma~\ref{lem:stronguppergradient3d}(iii),  the local slope $\vert \partial \phi \vert_{\mathcal{D}}$ is a strong upper gradient for $\phi$.
Thus,  Theorem~\ref{th:abstract convergence 2} implies that the time-discrete approximations converge to a $\tilde p$-curve of maximal slope  $y \colon [0,\infty) \to \mathscr{S}_{M}$   for $\phi$ with respect to $\vert \partial \phi \vert_{\mathcal{D}}$.  
In particular, the curve of maximal slope satisfies the balance
	\begin{align}\label{energybalanceused}
	\frac{1}{\tilde p} \int_0^T |y'|_{\mathcal{D}}^{\tilde p}(t) \, {\rm d} t + \frac{\tilde p - 1}{\tilde p} \int_0^T |\partial \phi|_{\mathcal{D}}^{\tilde p /(\tilde p -1)}(y(t)) \, {\rm d} t + \phi(y(T)) = \phi(y_0) \ \  \ \ \ \forall \,  T>0.
	\end{align} 
%The same strategy applies for the proof of (ii), recalling that Lemma~\ref{lem:repsmallstrain} addresses properties of the local slope for small strains.

\noindent \textit{Step 2:}   
As $y$ is a curve of maximal slope, we get that $\phi(y(t))$ is decreasing in time, see \eqref{energybalanceused}. This yields together with \eqref{aprioribounds} that 
\begin{align}
y \in L^\infty\big([0,\infty); W^{2,p}(\Omega;\R^d) \big). \label{boundedlimiting}
\end{align}  
Moreover, \eqref{energybalanceused} implies that $\vert y'\vert_{\mathcal{D}} \in L^{\tilde p}([0,\infty))$.  
As $\mathcal{D}$ is equivalent to the $W^{1,\tilde p}(\Omega)$-norm, see \eqref{equivalence},
we observe that
$y$ is an absolutely continuous curve with respect to $W^{1,\tilde p}(\Omega)$. 
Thus, by \cite[Remark~1.1.3]{AGS} and \eqref{energybalanceused} we observe that $y$ is differentiable for a.e.\ $t$ and we have $ y \in W^{1,\tilde p}([0,\infty) \times \Omega ;\R^d)$. More precisely, for all $0 \leq s <t$, and almost everywhere in $\Omega$   it holds that 
\begin{align}\label{lemma:fundamental}
y(t)-y(s) = \int_s^t \partial_t y(r) \,  {\rm d}r    .
\end{align}

\noindent \textit{Step 3:}  We now use a standard technique to relate curves of maximal slope to PDEs in Hilbert spaces, see \cite[Section 1.4]{AGS}. More precisely, the proof follows the lines of \cite[Theorem 2.2(ii)]{MFMKDimension}. For the sake of convenience, we again neglect $f$. This term  can be included  by standard adaptions.
Using Fatou's lemma, we can proceed as in \eqref{helpsbelow}, and we find for a.e.~$t$ that
\begin{align}\label{metricderivativeest}
|y'|_{\mathcal{D}}(t) = \lim_{s \to t} \frac{\mathcal{D}(y(s),y(t))}{|s-t|} \geq \left(\int_\Omega {\tilde p} R(\nabla y(t), \nabla \partial_t y (t)) \di x \right)^{1/{\tilde p}}.
\end{align}
We now determine the derivative $\frac{{\rm d}}{{\rm d} t} \phi(y(t))$. Due to \eqref{energybalanceused}, $|y'|_{\mathcal{D}}(t)$ and $|\partial \phi|_{\mathcal{D}} (y(t))$ are finite for a.e.~$t $. Fix such a $t \in [0,\infty)$. Then, Lemma~\ref{finerepresentation} implies that 
$\diver \mathcal{L}_ P(\nabla^2 y(t)) \in W^{-1,{\tilde p/ (\tilde p -1 )}}(\Omega)$.
Using the convexity of $P$, integration by parts, and a Taylor expansion, we deduce together with $1<\tilde p $, \eqref{aprioribounds}, and \eqref{equivalence} that 
\begin{align*}
 \frac{{\rm d}}{{\rm d} t} \phi(y(t)) & \geq \liminf\limits_{s \to t} \frac{1}{s-t} \int_\Omega \left( \partial_F W(\nabla y(t)): ( \nabla y(s) - \nabla y(t)) - \diver \partial_G P(\nabla^2 y(t)) : (\nabla y(s) - \nabla y(t))\right) \di x \\
 & \qquad - C \limsup\limits_{s \to t} \frac{  \mathcal{D} (y(s),y(t) )^{\tilde p} } { \vert s - t\vert^{\tilde p} } \vert s - t \vert^{\tilde p -1} .
\end{align*}
Let $\bar w(t)$ be the functions defined in Lemma~\ref{finerepresentation}.
Shortly writing \begin{align*}
A_{\bar w} &\defas A(x) (\nabla   \bar w(t))^T \nabla y(t) + A(x)(\nabla y(t))^T  \nabla  \bar w(t) \qquad \text{and} \qquad \\ A_{\partial_t y} &\defas A(x) (\partial_t  \nabla y(t))^T \nabla y(t) + A(x) (\nabla y(t))^T \partial_t  \nabla y(t),
\end{align*}  we find by \eqref{eulerlagrangefordifferentialop} and \eqref{explicit2} that
\begin{align}\label{derivativeenergyest}
\frac{{\rm d}}{{\rm d} t} \phi(y(t)) \geq \int_\Omega
\vert  A_{\bar w} \vert^{{\tilde p}-2}  A_{\bar w}:   A_{\partial_t y} 
 \di x.
\end{align}

\noindent \textit{Step 4:}   
Combining  \eqref{derivativeenergyest}, Young's inequality with powers $\tilde p / (\tilde p -1)$ and $\tilde p$, \eqref{finerepresentationformula}, \eqref{relationRandderivative}, and \eqref{metricderivativeest} implies that
\begin{align*}
 \frac{{\rm d}}{{\rm d} t} \phi(y(t)) 
  \geq - \int_\Omega \left( \frac{{\tilde p}-1}{{\tilde p}} \vert  A_{\bar w} \vert^{{\tilde p} } + {\tilde p}^{-1} \vert A_{\partial_t y} \vert^{\tilde p} \right) \di x   \geq - \frac{{\tilde p}-1}{{\tilde p}} \vert \partial \phi \vert_{\mathcal{D}}^{{\tilde p}/({\tilde p}-1)}(y(t))  - \frac{1}{{\tilde p}} |y'|_{\mathcal{D}}^{\tilde p}(t).
\end{align*}  
Since $y$ is a $\tilde p$-curve of maximal slope, see Definition~\ref{main def2}(iii), all inequalities are in fact equalities, and thus it holds that
\begin{align*}
\vert A_{\partial_t y} \vert^{\tilde p} = \vert  A_{\bar w} \vert^{{\tilde p}}  \qquad \text{ and } \qquad    A_{\bar w}:   A_{\partial_t y} = - \vert A_{\bar w}  \vert \vert A_{\partial_t y} \vert
\end{align*} 
a.e.~in $\Omega$. This gives $ \vert A_{\bar w} +  A_{\partial_t y} \vert^2 = 0$, and therefore multiplying $ A_{\partial_t y} = - A_{\bar w}$ with $2(\nabla y(t)) \vert A_{\partial_t y} \vert^{\tilde p -2}= 2(\nabla y(t)) \vert A_{\bar w} \vert^{\tilde p -2}$ from the left, taking the scalar product with $A(x) ( (\nabla \varphi) \nabla y + ( \nabla y )^T \nabla \varphi )$ for $\varphi \in W_0^{2,p} (\Omega;\R^d) \subset W_0^{1,\tilde p} (\Omega;\R^d)$, and using \eqref{explicit2} we get that
\begin{align*}
\partial_{\dot{F}}R(x,\nabla y(t), \partial_t  \nabla y(t)) : \nabla \varphi =
-\partial_{\dot{F}}R(x,\nabla y(t), \nabla   \bar w(t)) : \nabla \varphi
\end{align*}
a.e.~in $\Omega$.
Eventually, we can conclude \eqref{weaksolutiondef} by an integration of the latter identity, \eqref{eulerlagrangefordifferentialop}, and integration by parts.
% \begin{align*}
% \int_\Omega \big( \partial_F W(\nabla y(t)) +  \partial_{\dot{F}}R(x,\nabla y(t), \partial_t  \nabla y(t))  \big) : \nabla \varphi + \partial_G P(x,\nabla^2 y(t)) : \nabla^2 \varphi \di x = 0.
% \end{align*}
\end{proof}

\begin{rem}[Energy balance]\label{explicitexpression}
In the proof of Theorem~\ref{maintheorem1}(ii) we have seen that the sharp lower bound on the metric derivative in \eqref{metricderivativeest} is actually an equality and that $|y'|_{\mathcal{D}}^{\tilde p}(t) = |\partial \phi|_{\mathcal{D}}^{\tilde p /(\tilde p -1)}(y(t))$ for a.e.~$t$. 
Thus, the energy-dissipation-balance in \eqref{energybalanceused} implies that a $\tilde p$-curve of maximal slope $y$ satisfies
 
\begin{align*}
\tilde p \int_0^T \int_\Omega  R(\nabla y(t), \nabla \partial_t y (t)) \di x \di t = \phi( y_0) - \phi( y(T) )
\end{align*}
for $T >0$.

\end{rem}

\section{Long-time behavior for small strains}\label{sec:longtime}

This section is devoted to the proof of Theorem~\ref{maintheorem3}, which is inspired by \cite[Theorem~2.4.15]{AGS}.

\begin{proof}[Proof of Theorem~\ref{maintheorem3}]
  The uniqueness of a minimizer of $\phi$ follows from the strict convexity in \linebreak Lemma~\ref{lem:localconv}(ii), concluding the proof of (i). 
  Before proving (ii), (iii), and (iv), we provide a preliminary step addressing properties related to the local slope $\vert \partial \phi \vert_{\mathcal{D}}(y^\delta(t))$.
  To this end, we evaluate the local slope at $y^\delta(t)$ and compare the supremum in \REVB Lemma~\ref{lem:repsmallstrain} \END with the minimizer $y^\delta_\infty$.
This implies together with \eqref{aprioribounds} that
\begin{align}\label{ineq:slopecompare}
&\phi(y^\delta(t)) - \phi(y^\delta_\infty)  + \tfrac{1}{2}    \lambda     \int_\Omega \vert \nabla y^\delta(t) - \nabla  y^\delta_\infty \vert^2 \di x   \notag \\ &\qquad   \leq \mathcal{D}(y^\delta(t),y^\delta_\infty) (1 + C \Vert \nabla y^\delta_\infty - \nabla y^\delta(t) \Vert_{L^\infty(\Omega) }^{\tilde p-1} + C \Vert \nabla y^\delta_\infty - \nabla y^\delta(t) \Vert_{L^\infty(\Omega) })^{1/\tilde p}  \vert \partial \phi \vert_{\mathcal{D}}(y^\delta(t)) \notag  \\
&\qquad   \leq C\mathcal{D}(y^\delta(t),y^\delta_\infty)    \vert \partial \phi \vert_{\mathcal{D}}(y^\delta(t)) 
\end{align}
Since $y^\delta$ is a ${\tilde p}$-curve of maximal slope with respect to the strong upper gradient $\vert \partial \phi \vert_{\mathcal{D}}$, we obtain that $\vert \partial \phi \vert_{\mathcal{D}}(y^\delta(t)) \to 0$ as $t \to +\infty$. Since the left-hand side of \eqref{ineq:slopecompare} is nonnegative due to the minimality of $y^\delta_\infty$ and the fact that $\lambda>0$, we obtain $\phi(y^\delta(t)) \to \phi(y^\delta_\infty) $ as $t \to \infty$. Thus, \REV $y^\delta(t) \rightharpoonup y^\delta_\infty$ weakly in $W^{2,p}(\Omega)$ as $t \to \infty$, \END see Lemma~\ref{th: metric space}(ii). 
Moreover, as $y^\delta$ is a ${\tilde p}$-curve of maximal slope with respect to the strong upper gradient $\vert \partial \phi \vert_{\mathcal{D}}$, we get for a.e.~$t\in [0,T]$ the equality
\begin{align*}
\vert \partial \phi \vert_{\mathcal{D}}(y^\delta(t)) \vert (y^{\delta})^{\prime}\vert_{\mathcal{D}}(t) = \frac{\tilde p -1}{\tilde p} \vert \partial \phi \vert_{\mathcal{D}}^{\tilde p / (\tilde p -1)}(y^\delta(t)) + \frac{1}{\tilde p}  \vert (y^{\delta})^{\prime}\vert^{\tilde p}_{\mathcal{D}}(t).
\end{align*} Thus, we have an equality in Young's inequality, which implies together with \eqref{maximalslope} that
\begin{align}\label{eq:derivativeenergy}
- \frac{\di}{\di t} \phi (y^\delta(t)) = \vert \partial \phi \vert_{\mathcal{D}}^{\tilde p / (\tilde p -1)}(y^\delta(t)).
\end{align}

Eventually, we address the long-time behavior.
First, we set $\tilde p = 2$.
By \eqref{ineq:slopecompare},  \eqref{equivalence},  Young's inequality with power $2$, and the positivity of $\lambda$ for $\delta$ sufficiently small it holds that  
\begin{align}
\phi(y^\delta(t)) - \phi(y^\delta_\infty)   
\leq    C \vert \partial \phi \vert_{\mathcal{D}}^{2}(y^\delta(t))    . \label{estimateslope3}
\end{align}
In view of \eqref{eq:derivativeenergy},  we thus get
\begin{align*}
- C  \left( \phi(y^\delta(t)) - \phi(y^\delta_\infty)  \right) 
&\geq  \frac{\di}{\di t} \left( \phi (y^\delta(t)) - \phi (y^\delta_\infty) \right).
\end{align*}
Thanks to Gronwall's inequality, we find
\begin{align}\label{conclusiongronwall}
 \phi (y^\delta(t)) - \phi (y^\delta_\infty)\leq \exp (- C   t) \left(\phi (y^\delta(0)) - \phi (y^\delta_\infty) \right).
\end{align} 
Eventually, combining Remark~\ref{explicitexpression} with \eqref{conclusiongronwall} leads to
\begin{align*}
2 \int_0^T \int_\Omega   R(\nabla y(t), \nabla \partial_t y (t)) \di x \di t \geq (1-\exp(-CT)) \big( \phi( y^\delta_0) - \phi(y^\delta_\infty)\big) . 
\end{align*}
This concludes the proof of (ii).

Let now $\tilde p \in (1,2)$. By \REV Jensen's \END inequality, \eqref{ineq:slopecompare},  \eqref{equivalence},  Young's inequality with power $2$, and the positivity of $\lambda$ for $\delta$ sufficiently small we again see that
\eqref{estimateslope3} holds. Thus, \eqref{eq:derivativeenergy} and \eqref{estimateslope3} yield
\begin{align*}
	- C \left( \phi(y^\delta(t)) - \phi(y^\delta_\infty)   \right)^{ \tilde p/(2\tilde p -2)}
\geq    \frac{\di}{\di t} \left( \phi(y^\delta(t)) - \phi(y^\delta_\infty)   \right).
\end{align*}
To solve this differential inequality, we set $f(t) \defas \phi (y^\delta(t)) - \phi (y^\delta_\infty)$ and observe that $f$
is nonnegative and decreasing. Moreover, let $s \defas \tilde p/(2\tilde p -2)>1$. Using integration by substitution,
 we find that
\begin{align}\label{differentialineq}
	\frac{1}{1-s} f(t)^{1-s} - \frac{1}{1-s} f(t_0)^{1-s} = \int_{f(t_0)}^{f(t)} x^{-s} \di x = \int_{t_0}^{t} f(r)^{-s} f'(r) \di r  \leq - C (t- t_0)
\end{align}
for all $t \geq t_0 \geq 0$.
By taking the $(s-1)$th root, we find
\begin{align}\label{solution}
	f(t) \leq   \left(   f(t_0)^{1-s} - C (1-s) (t-t_0) \right)^{\frac{1}{1-s}}.
\end{align}
Setting $t_0 = 0$, this shows (iii).

Let now $\tilde p >2$. By the Gagliardo–Nirenberg interpolation inequality we find
\begin{align}
\left(\int_\Omega \vert \nabla y^\delta(t) - \nabla y^\delta_\infty \vert^{\tilde p} \di x  \right)^{1/\tilde p} & \leq C \Vert  \nabla^2 y^\delta(t) - \nabla^2 y^\delta_\infty \Vert_{L^p(\Omega)}^\alpha \Vert \nabla y^\delta(t) - \nabla y^\delta_\infty  \Vert_{L^2(\Omega)}^{1-\alpha} \notag \\ &\quad + C\Vert \nabla y^\delta(t) - \nabla y^\delta_\infty  \Vert_{L^2(\Omega)}, \label{Help2}
\end{align}
where
\begin{align*}
\frac{1}{\tilde p} = \alpha \left(\frac{1}{p} - \frac{1}{d} \right) + (1-\alpha) \frac{1}{2}.
\end{align*}
As $p>d$, rearranging yields
\begin{align}
1-\alpha = \frac{\frac{1}{\tilde p} - \frac{1}{p} + \frac{1}{d}}{\frac{1}{2} - \frac{1}{p} + \frac{1}{d}} > \frac{\frac{1}{\tilde p}}{\frac{1}{2}} = \frac{2}{\tilde p}, 	\qquad \frac{2}{1-\alpha} < \tilde p, \qquad 2 > \frac{2}{1+\alpha} > \frac{\tilde p}{\tilde p -1}. \label{exponentineq}
\end{align}
By \eqref{ineq:slopecompare}, \eqref{equivalence}, \eqref{Help2}, Young's inequality with powers $\frac{2}{1-\alpha}$ and $\frac{2}{1+\alpha}$ and power $2$, and the fact that $\vert \partial \phi \vert_{\mathcal{D}}(y^\delta(t)) \to 0$ for $t \to \infty$ there exists $t_0>0$ such that
\begin{align*}
\phi(y^\delta(t)) - \phi(y^\delta_\infty)    \leq C \vert \partial \phi \vert_{\mathcal{D}}(y^\delta(t))^{2/(1+\alpha)}  + C \vert \partial \phi \vert_{\mathcal{D}}(y^\delta(t))^2  \leq 2C \vert \partial \phi \vert_{\mathcal{D}}(y^\delta(t))^{2/(1+\alpha)}
\end{align*}
for all $t \geq t_0$.
In view of \eqref{eq:derivativeenergy},  we thus get
\begin{align*}
- C  \left( \phi(y^\delta(t)) - \phi(y^\delta_\infty)  \right)^{\frac{\tilde p}{\tilde p -1} \frac{1+\alpha}{2}}
\geq  \frac{\di}{\di t} \left( \phi (y^\delta(t)) - \phi (y^\delta_\infty) \right).
\end{align*}
Due to \eqref{exponentineq}, the exponent $s \defas \frac{\tilde p}{\tilde p -1} \frac{1+\alpha}{2}$ on the left-hand side is smaller than $1$.
Following the lines of \eqref{differentialineq}, we can solve this differential inequality, and
take the $(1-s)$th root, implying that \eqref{solution} holds,
as long as the base on the right-hand side is positive.
Since the left-hand side is positive, there exists a time $T_{\rm ext}>0$ such that $f(T_{\rm ext}) = 0$. As $\phi( y^\delta( \REV \cdot \END ))$ is decreasing, we have $\phi (y^\delta(t)) = \phi (y^\delta_\infty)$ for all $t \geq T_{\rm ext}$. Finally, we conclude the proof using (i).
%
%
% Then, \eqref{ineq:slopecompare},\eqref{equivalence} and Young's inequality with powers $\tilde p$ and $\tilde p /(\tilde p-1)$ imply together with \eqref{aprioribounds} that
\end{proof}

\begin{rem}\label{remarknotsharp}
	To obtain the decay estimate for $\tilde p > 2$, we benefit from the integrability of the second gradient: by using the Gagliardo-Nirenberg interpolation inequality in \eqref{Help2}, we can control the $L^{\tilde p}(\Omega)$-norm by means of the $L^2(\Omega)$-norm with a suboptimal scaling. 
	Thus, we \REV do \END not match the heuristic scaling with $s = \frac{\tilde p}{2 \tilde p -2}$ in \eqref{solution}. Still, the exponent is sufficiently small, showing that equilibrium is reached in finite time.
\end{rem}

\section*{Acknowledgements} 
This work was funded by  the DFG project FR 4083/5-1 and  by the Deutsche Forschungsgemeinschaft (DFG, German Research Foundation) under Germany's Excellence Strategy EXC 2044 -390685587, Mathematics M\"unster: Dynamics--Geometry--Structure. 
The work was further supported by the DAAD project 57600633 / DAAD-22-03 and by the Deutsche Forschungsgemeinschaft (DFG, German Research Foundation) under Germany’s Excellence Strategy – EXC-2047/1 – 39068581.

 \typeout{References}

\end{document}